\RequirePackage{fix-cm}
\documentclass[11pt]{article}

\usepackage{geometry}
\usepackage{amsmath, amsthm, amssymb, dsfont, mathrsfs, graphicx, caption, float, multirow, upgreek}

\usepackage{fouriernc, anyfontsize}
\usepackage[T1]{fontenc}

\usepackage{hyperref}
\usepackage[dvipsnames]{xcolor}
\hypersetup{colorlinks=true, linkcolor=Emerald, citecolor=Emerald}

\allowdisplaybreaks

\usepackage{fancyhdr}
\newtheoremstyle{theorem}
  {10pt}
  {10pt}
  {\sl}
  {\parindent}
  {\bf}
  {. }
  { }
  {}
\theoremstyle{theorem}
\newtheorem{theorem}{Theorem}
\newtheorem{lemma}{Lemma}
\newtheorem{corollary}{Corollary}
\newtheorem{definition}{Definition}

\begin{document}
\title{\bf Set-valued stochastic integrals for \\ convoluted L\'{e}vy processes}
\author{Weixuan Xia\thanks{University of Southern California, Department of Mathematics, Los Angeles, CA 90089, USA. \newline\indent Email: weixuanx@usc.edu}}
\date{2023}

\maketitle

\thispagestyle{plain}

\begin{abstract}
  In this paper we study set-valued Volterra-type stochastic integrals driven by L\'{e}vy processes. Upon extending the classical definitions of set-valued stochastic integral functionals to convoluted integrals with square-integrable kernels, set-valued convoluted stochastic integrals are defined by taking the closed decomposable hull of the integral functionals for generic time. We show that, aside from well-established results for set-valued It\^{o} integrals, while set-valued stochastic integrals with respect to a finite-variation Poisson random measure are guaranteed to be integrably bounded for bounded integrands, this is not true when the random measure is of infinite variation. For indefinite integrals, we prove that it is a mutual effect of kernel singularity and jumps that the set-valued convoluted integrals are possibly explosive and take extended vector values. These results have some important implications on how set-valued fractional dynamical systems are to be constructed in general. Two classes of set-monotone processes are studied for practical interests in economic and financial modeling. \medskip\\
  \textsc{MSC2020 Classifications:} 28B20; 60G22; 60G57 \medskip\\
  \textsc{Key Words:} Set-valued stochastic processes; infinite-variation jumps; Volterra-type integrals; singular kernels
\end{abstract}

\newcommand{\dd}{{\rm d}}
\newcommand{\pd}{\partial}
\newcommand{\ii}{{\rm i}}
\newcommand{\Gf}{\mathrm{\Gamma}}
\newcommand{\PP}{\mathbb{P}}
\newcommand{\E}{\mathbb{E}}
\newcommand{\cl}{\mathrm{cl}}
\newcommand{\co}{\mathrm{co}}
\newcommand{\dec}{\mathrm{dec}}
\newcommand{\card}{\mathrm{card}}
\newcommand{\0}{\mathbf{0}}

\medskip

\section{Introduction}\label{sec:1}

Set-valued stochastic integrals have been studied extensively for almost three decades. Its first appearance is arguably in [Aumann, 1965] \cite{A2} to define the Lebesgue integral of set-valued functions (multifunctions) motivated by the study of economic preferences without the completeness axiom. Later, [Kisielewicz, 1997] \cite{K3} was the first to define a set-valued It\^{o} integral functional as a subset of an $\mathds{L}^{2}$ space, under the terminology ``sub-trajectory integrals.'' Since then, many works have been devoted to the study of Aumann stochastic integrals and set-valued It\^{o} integrals. In particular, formal definitions of such integrals were presented in [Li and Li, 2009a] \cite{LL1} and [Kisielewicz, 2012] \cite{K5} in terms of the decomposable hulls of the associated integral functionals, based on the notion of measurable selectors. In [Li et al., 2010] \cite{LLL}, set-valued integrals were defined with respect to a square-integrable continuous martingale and [Malinowski, 2013] \cite{M1} also considered set-valued integrals driven by semimartingales, where it was argued that usual definitions result in a discrepancy between the integrals driven by the whole semimartingale and the Minkowski sum of those driven by the finite-variation process and the local martingale in the It\^{o} decomposition.

The interest in studying set-valued stochastic processes principally stems from dynamical systems where the exact values of parameters or mechanisms possess uncertainty or ambiguity. Up until now, plenty of research has been conducted to deal with set-valued systems with uncertainty (e.g., driven by diffusion), by means of set-valued stochastic differential equations or stochastic differential inclusions which are natural generalizations of single-valued stochastic differential equations; we mention, among others, [Kisielewicz, 1991] \cite{K1}, [Kisielewicz, 1993] \cite{K2}, [Kisielewicz, 1997] \cite{K3}, [Kisielewicz, 2005] \cite{K4}, [Zhang et al., 2009] \cite{ZLMO}, [Li and Li, 2009b] \cite{LL2}, [Ren and Wu, 2011] \cite{RW}, and also [Malinowski, 2015] \cite{M2}. In particular, it is noteworthy that for such equations to be well-defined based on the definitions of the set-valued stochastic integrals, integrable boundedness turns out to be a crucial requirement. Unfortunately, in the papers [Michta, 2015] \cite{M5} and [Kisielewicz, 2020a] \cite{K7} it was argued that set-valued It\^{o} integrals, unlike Aumann integrals, are not integrably bounded unless they are a singleton, regardless of whether the integrands are deterministic or stochastic. Such properties have undoubtedly hindered the study of the aforementioned dynamical systems driven by a Brownian motion. Nevertheless, it was later considered in [Kisielewicz and Michta, 2017] \cite{KM} that set-valued It\^{o} integrals can be made integrably bounded if the integrands are chosen to be non-decomposable subsets, most generally as the closed convex hull (taken in function spaces) of absolutely summable processes, which consequently make it possible to continue studying stochastic differential inclusions with non-single-valued It\^{o} integrals. For a more comprehensive treatment of set-valued stochastic integrals we refer to the recent book [Kisielewicz, 2020b] \cite{K8}. In the meantime, [Zhang et al., 2013] \cite{ZMO} defined and studied set-valued stochastic integrals with respect to a Poisson random measure, where attention was paid to finite intensity measures (i.e., processes of the compound Poisson type).

Aside from other areas that utilize dynamical systems, in mathematical finance, set-valued stochastic integrals find applications in analyzing risk measures in the presence of multiple risk factors, allowed to evolve over time according to preset mechanisms. We highlight [Ararat and Feinstein, 2021] \cite{AF}, which studied multi-valued risk measures via set-valued backward stochastic difference equations in a discrete-time setting, and tools for studying related problems in more general continuous-time settings, using set-valued backward stochastic differential equations have recently been developed -- see [Ararat et al., 2023] \cite{AMW} and [Ararat and Ma, 2023+] \cite{AM}. In addition, set-valued stochastic integrals are useful for incorporating stochastically evolving model ambiguity into asset price dynamics in financial markets, e.g., by generalizing the setup in [Liang and Ma, 2020] \cite{LM}, which has imposed such ambiguity up to the level of L\'{e}vy--Khintchine triplets. Noteworthily, the market need not be Markovian and empirical evidence suggests short- or long-term memory effects, which can be captured by introducing frictions via convoluted processes; see, e.g., the pioneering work [Gatheral et al., 2018] \cite{GJR} for volatility modeling.

Indeed, over recent decades, convoluted L\'{e}vy processes that are obtained by integrating a suitable kernel function against a L\'{e}vy process have been studied in great depth, motivated from long-range dependence in randomness sources. Such processes include the classical fractional Brownian motion (see [Nuarlart, 2003] \cite{N} for an overview) and the two-sided fractional L\'{e}vy processes proposed in [Marquardt, 2006] \cite{M4}; see also [Wolpert and Taqqu, 2004] \cite{WT} and [Barndorff-Nielsen et al., 2014] \cite{B-NBPV}. A formal treatment of the so-called ``convoluted L\'{e}vy processes'' was given in [Bender and Marquardt, 2008] \cite{BM}, which considered their calculus, as well as [Tikanm\"{a}ki and Mishura, 2011] \cite{TM}, which considered various integral transformations of L\'{e}vy processes to construct convoluted processes and studied their regularity properties.

Combining the above aspects, it is natural to consider set-valued convoluted L\'{e}vy processes and more generally set-valued convoluted stochastic integrals, which have not been formally defined and studied so far. By doing so one also has the potential to consider model ambiguity in parameters associated with memory, such as fraction indices, in stochastic environments.

Moreover, in multi-utility maximization problems under imprecise tastes, set-valued integrals can also be applied to describe the general evolution of a utility set (see [Xia, 2023+] \cite{X}), when it is allowed to be time-varying subject to various dynamic factors. Therefore, if the market is driven by convoluted L\'{e}vy processes to allow for short- or long-range dependence and jumps in returns as mentioned above, it is crucial to define their set-valued integration before preferential changes may be modeled with these advanced characteristics as well.

In this paper our focus will be on the construction of set-valued convoluted stochastic integrals under square-integrable kernels, as an adequate generalization of classical stochastic integrals studied in [Kisielewicz, 2020b] \cite{K8} and [Zhang et al., 2013] \cite{ZMO}. Most importantly, we will consider integrals with respect to infinite Poisson random measures and give some of its unexplored properties such as integrable unboundedness (comparable to the It\^{o} case). It is worth mentioning that this type of integrals are not conceptually substitutable by finite measure-driven ones plus It\^{o} integrals. For example, the activity index ([Blumenthal and Getoor, 1961] \cite{BG}) of a semimartingale with infinite-variation jumps strictly lies between 1 and 2 and can be of the essence in depicting the behavior of financial time series observed at high frequencies (see, e.g., [Todorov and Tauchen, 2011] \cite{TT} and also [Wang and Xia, 2022] \cite{WX} in the case of volatility modeling). Last but not least, we shall as well investigate the effect of kernel singularity on the integrability and explosiveness of the convoluted integrals to understand how related models should be written for potential applications.

\medskip

\section{Preliminaries}\label{sec:2}

To begin with, we synthesize some basic concepts and properties related to L\'{e}vy processes and set-valued random variables. These concepts are kept to line with the main results to be stated and practical interests as closely as possible. For more general definitions and detailed properties we refer to [Applebaum, 2009] \cite{A1} and [Lyasoff, 2017, Chapters 15--16] \cite{L} for L\'{e}vy processes as well as [Fryszkowski, 2004] \cite{F} and [Kisielewicz, 2020b, Chapters I--III] \cite{K8} for set-valued random variables.

\subsection{L\'{e}vy processes and convoluted integrals}\label{sec:2.1}

Let $(\Omega,\mathcal{F},\PP;\mathbb{F}\equiv(\mathscr{F}_{t})_{t\in[0,T]})$ be a complete filtered (separable) probability space with $T>0$ fixed, where the filtration $\mathbb{F}$ satisfies the usual conditions, with $\mathscr{F}_{T}=\mathcal{F}$, and on which we define $\xi\equiv(\xi_{t})_{t\in[0,T]}$ to be a general $d$-dimensional L\'{e}vy process admitting the following L\'{e}vy--It\^{o} decomposition:
\begin{equation}\label{2.1}
  \xi_{t}=\mu t+\sigma W_{t}+\int_{\|z\|\geq1}zN([0,t],\dd z)+\int_{0<\|z\|<1}z\tilde{N}([0,t],\dd z),\quad t\in[0,T],
\end{equation}
where $\mu\in\mathds{R}^{d}$ is the drift parameter, $\sigma\in\mathds{R}^{d\times m}$ is the Brownian dispersion parameter, $W\equiv(W_{t})$ is an $m$-dimensional Brownian motion and $N$ is a Poisson random measure on $(\mathds{R}^{d}\setminus\{\0\},[0,T])$ with L\'{e}vy (intensity) measure $\nu$ such that $\int_{\mathds{R}^{d}\setminus\{\0\}}(1\wedge\|z\|^{2})\nu(\dd z)<\infty$. By convention $N$ is independent from $W$, and we denote by $\tilde{N}\equiv N-\nu\times\mathrm{Leb}_{[0,T]}$ the corresponding compensated Poisson random measure.

Throughout the paper, we use the subscript-less notation $\|\;\|$ for the Euclidean norm of vectors and $\|\;\|_{\rm F}$ for the Frobenius norm of matrices. For $E$ a Euclidean space, $D$ a subset of $[0,T]\times\mathds{R}^{d}$, and $M$ a $\upsigma$-finite measure on $[0,T]\times\mathds{R}^{d}$, we shall denote by $\mathds{L}^{p}(D\times\Omega,\mathcal{G},M\times\PP;E)$ ($p\geq1$) the space of all equivalence (under $M\times\PP$-a.e.) classes of $p$-integrable functions $X:D\times\Omega\mapsto E$, where $\mathcal{G}$ is some sub-$\upsigma$-algebra of $\mathcal{B}([0,T])\otimes\mathcal{F}$. It is understood that $\mathds{L}^{p}(D\times\Omega,\mathcal{G},M\times\PP;E)$ is for every $p\geq1$ a normed space equipped with the norm
\begin{equation*}
  \|X\|_{\mathds{L}^{p}}=\bigg(\E\int_{D}\|X\|^{p}\dd M\bigg)^{1/p}.
\end{equation*}
Whenever the topology and measures are trivially understood, the notation is often abbreviated as $\mathds{L}^{p}(\; ;\;)$.

Besides, if $\nu(\{z\in\mathds{R}^{d}:0<\|z\|<1\})<\infty$ then the Poisson random measure $N$ is said to be of finite activity, in which case it represents a homogeneous Poisson process; otherwise, it is of infinite activity. Further, if $\int_{0<\|z\|<1}\|z\|\nu(\dd z)<\infty$ then $N$ is said to be of finite variation, or else it is of infinite variation. In the infinite-variation case, the general decomposition (\ref{2.1}) stands, and our interest will be more on the fourth term of (\ref{2.1}) with respect to the compensated measure.

We shall use the following definition for Volterra-type integrals.

\begin{definition}\label{def:1}
A function $g:\{(t,s):t\in(0,T],s\in[0,t)\}\mapsto\mathds{R}$ is said to be a (suitable) kernel of Volterra type if: \medskip\\
(i) $g(t,\cdot):[0,t)\mapsto\mathds{R}$ is measurable for every $t\in(0,T]$; \\
(ii) $g(t,\cdot)\not\equiv0$ and $g(t,\cdot)\in\mathds{L}^{2}([0,t);\mathds{R})$, for every $t\in(0,T]$; \\
(iii) $g$ is continuously differentiable in the domain $\{(t,s):t\in(0,T),s\in(0,t)\}$.
\end{definition}

The Volterra-type kernel $g$ is called stationary if $g(t,s)=g(t-s)$ for every $t\in(0,T]$ and $s\in[0,t)$. If $\lim_{s\nearrow t}|g(t,s)|=\infty$ for some $t\in(0,T]$, the kernel is said to be singular at time $t$. Clearly, for stationary kernels, singularity is only meaningful for the entire domain $[0,T]$.

With such a (suitable) kernel $g$ defined above, a $g$-convoluted L\'{e}vy process is given by the following Volterra-type integral:
\begin{align*}
  \xi^{(g)}_{t}&:=\int^{t}_{0}g(t,s)\dd\xi_{s}\\
  &=\mu\int^{t}_{0}g(t,s)\dd s+\sigma\int^{t}_{0}g(t,s)\dd W_{s}+\int^{t}_{0}\int_{\|z\|\geq1}g(t,s)zN(\dd s,\dd z)\\
  &\qquad+\int^{t}_{0}\int_{0<\|z\|<1}g(t,s)z\tilde{N}(\dd s,\dd z),\quad t\in(0,T].
\end{align*}
We will stick to the superscript notation $(g)$ to indicate the convoluting kernel. Some well-known examples are given below, with commonly used kernels of Volterra type that only take positive values.

\medskip

\textsl{Example 1.}\quad The exponential kernel is given for a parameter $\kappa>0$ by
\begin{equation*}
  g(t,s)\equiv g(t-s)=e^{-\kappa(t-s)},
\end{equation*}
which is clearly stationary and never singular. In this case $\xi^{(g)}$ is none but a L\'{e}vy-driven Ornstein--Uhlenbeck process.

\medskip

\textsl{Example 2.}\quad The Riemann--Liouville kernel, which sits at the cornerstone of fractional calculus, reads
\begin{equation*}
  g(t,s)\equiv g(t-s)=\frac{(t-s)^{\beta-1}}{\Gf(\beta)},
\end{equation*}
for a parameter $\beta>1/2$ and where $\Gf$ denotes the usual gamma function. Obviously, $g$ is stationary and it is singular if and only if $\beta\in(1/2,1)$. In this case $\xi^{(g)}$ is a one-sided fractional L\'{e}vy process (see, e.g., [El Euch and Rosenbaum, 2018, Equation (1.3)] \cite{ER}).

\medskip

\textsl{Example 3.}\quad The Molchan--Golosov kernel ([Molchan and Golosov, 1969] \cite{MG}) reads for $\beta>1/2$
\begin{equation*}
  g(t,s)=(t-s)^{\beta-1}\;_{2}\mathrm{F}_{1}\bigg(-\beta,\beta-1;\beta;-\frac{t-s}{t}\bigg),
\end{equation*}
where $\;_{2}\mathrm{F}_{1}$ is the Gauss hypergeometric function. Then, $g$ is not stationary but can be shown to be singular if and only if $\beta\in(1/2,1)$. With such a kernel, if $N\equiv0$ then it has been shown in [Jost, 2006] \cite{J} that $\xi^{(g)}$ is equivalent to a two-sided (i.e., real-valued time) fractional Brownian motion.

\medskip

There are also many other hybrid kernels that result from mixing in various ways the exponential and Riemann--Liouville types to gain practical interests, and we refer to [Wolpert and Taqqu, 2004] \cite{WT} and [Wang and Xia, 2022] \cite{WX}, among others.

Likewise, since $\int^{t}_{0}g^{2}(t,s)\dd s<\infty$ for any $t\in(0,T]$, we can consider a general Volterra-type stochastic integral with respect to time, the Brownian motion $W$, and the Poisson random measure (compensated measure) $N$ ($\tilde{N}$), i.e., a process of the form\footnote{Of course, the kernel can vary across different integrators. The formulation is adopted for succinctness as focus is currently on the general effect of convolution.}
\begin{align}\label{2.2}
  X^{(g)}_{t}&=X_{0}+\int^{t}_{0}g(t,s)f_{1}(s-)\dd s+\int^{t}_{0}g(t,s)f_{2}(s-)\dd W_{s}+\int^{t}_{0}\int_{\|z\|\geq1}g(t,s)f_{3}(s-,z)N(\dd s,\dd z) \nonumber\\
  &\qquad+\int^{t}_{0}\int_{0<\|z\|<1}g(t,s)f_{4}(s-,z)\tilde{N}(\dd s,\dd z),\quad t\in(0,T],
\end{align}
where $X_{0}$ is $\mathscr{F}_{0}$-measurable, and $f_{q}$, for $q\in\{1,2,3,4\}$, are predictable processes over $[0,T]$ such that $(gf_{1})(t,\cdot)\in\mathds{L}^{p}\big([0,t)\times\Omega,\mathcal{P}(\mathbb{F}),\mathrm{Leb}_{[0,t)}\times\PP;\mathds{R}^{d}\big)$ for some $p\geq1$, $(gf_{2})(t,\cdot)\in\mathds{L}^{2}\big([0,t)\times\Omega,\mathcal{P}(\mathbb{F}),\mathrm{Leb}_{[0,t)}\times\PP;\mathds{R}^{d\times m}\big)$, $(gf_{3})(t,\cdot)\in\mathds{L}^{2}\big([0,t)\times(\mathds{R}^{d}\setminus\{\0\})\times\Omega,\mathcal{P}(\mathbb{F}), \mathrm{Leb}_{[0,t)}\times\nu\upharpoonright_{\{\|z\|\geq1\}}\times\PP;\mathds{R}^{d}\big)$, and $(gf_{4})(t,\cdot)\in\mathds{L}^{2}\big([0,t)\times(\mathds{R}^{d}\setminus\{\0\})\times\Omega,\mathcal{P}(\mathbb{F}), \mathrm{Leb}_{[0,t)}\times\nu\upharpoonright_{\{0<\|z\|<1\}}\times\PP;\mathds{R}^{d}\big)$, for every $t\in(0,T]$, where $\mathcal{P}(\mathbb{F})$ denotes the predictable sub-$\upsigma$-algebra, namely the $\upsigma$-algebra generated by all left-continuous $\mathbb{F}$-adapted processes, and $\upharpoonright$ denotes measure restriction. In particular, $f_{q}$'s can be required to satisfy the conditions
\begin{align}\label{2.3}
  &\sup\{\E\|f_{q}(t)\|^{p}:t\in[0,T]\}<\infty,\quad q\in\{1,2\}, \nonumber\\
  &\sup\bigg\{\E\int_{\|z\|\geq1}\|f_{q}(t,z)\|^{p}\nu(\dd z):t\in[0,T]\bigg\}<\infty,\quad q=3, \nonumber\\
  &\sup\bigg\{\E\int_{0<\|z\|<1}\|f_{q}(t,z)\|^{p}\nu(\dd z):t\in[0,T]\bigg\}<\infty,\quad q=4,
\end{align}
where $p=2$ for $q\in\{2,4\}$. These basically ensure that we are considering integrands in the targeted $\mathds{L}^{p}$ spaces as for non-convoluted integrals.

\subsection{Set-valued random variables}\label{sec:2.2}

For the Euclidean space $E$, $\mathcal{S}(E)$ denotes the space of all nonempty subsets of $E$ and $\mathrm{Cl}(E)$ denotes the space of nonempty closed subsets of $E$. A set-valued random variable is a multifunction $X:\Omega\mapsto\mathrm{Cl}(E)$ that is $\mathcal{F}$-measurable, i.e., with the inverse $X^{-1}(A):=\{\omega\in\Omega:X(\omega)\cap A\neq\emptyset\}\in\mathcal{F}$ for every open subset $A\subset E$. By convention we write $\cl$ and $\co$ for the closure and convex hull operators, respectively, which are by default taken in the corresponding Euclidean space if there is no attached subscript; for normed spaces we can impose $\overline{\co}\equiv\cl\co$. For such an $X$, $\cl X$ and $\co X$ are both measurable if $X$ is.

A selector $f$ of the set-valued random variable $X$ is one such that $f(\omega)\in X(\omega)$ for $\PP$-\text{a.e.} $\omega\in\Omega$, whose existence follows from the Zermelo axiom of choice. Moreover, by the Kuratowski--Ryll--Nardzewski theorem, $X$ admits a measurable selector $f$ (i.e., a single-valued random variable) such that $f\in X$, $\PP$-a.s.

The following useful lemma follows from [Kisielewicz, 2020b, Chapter II Theorem 2.2.3] \cite{K8} and the separability of $E$.

\begin{lemma}\label{lem:1}
For a measurable set-valued random variable $X:\Omega\mapsto\mathrm{Cl}(E)$, there exists a sequence $\{f_{n}:n\in\mathds{N}_{++}\}\subset X$ of single-valued random variables (measurable selectors of $X$) such that $X=\cl\{f_{n}:n\in\mathds{N}_{++}\}$, $\PP$-a.s.
\end{lemma}

For $A,B\in\mathrm{Cl}(E)$, the Hausdorff distance (metric) is denoted as
\begin{equation*}
  \mathbf{d}_{\rm H}(A,B):=\max\Big\{\sup_{a\in A}\inf_{b\in B}\|a-b\|,\sup_{b\in B}\inf_{a\in A}\|a-b\|\Big\},
\end{equation*}
and $(\mathrm{Cl}(E),\mathbf{d}_{\rm H})$ forms a complete metric space.

Given a sub-$\upsigma$-algebra $\mathcal{G}\subset\mathcal{F}$, a subset $K$ of the $\mathds{L}^{p}(D\times\Omega,\mathcal{G},M\times\PP;E)$ space is said to be ($\mathcal{G}$-)decomposable if for every $f,l\in K$ and $A\in\mathcal{G}$ it holds that $\mathds{1}_{A}f+\mathds{1}_{A^{\complement}}l\in K$ (with complement taken in $\Omega$). We write $\dec K$ the decomposable hull of $K$ in $\mathds{L}^{p}$ and $\overline{\dec}K$ its closure in $\mathds{L}^{p}$.

For the set-valued random variable $X$, $S^{p}_{\mathcal{G}}(X):=\{f\in\mathds{L}^{p}(\Omega;E):f\in X,\;\PP\text{-a.s.}\}$ denotes the collection of $\mathcal{G}$-measurable $p$-integrable selectors of $X$. The set-valued random variable $X$ is said to be $p$-integrable as long as $S^{p}_{\mathcal{F}}(X)\neq\emptyset$; it is said to be $p$-integrably bounded if there exists $h\in\mathds{L}^{p}(\Omega;\mathds{R}_{+})$ such that $\mathbf{d}_{\rm H}(X,\{\0\})\leq h$, $\PP$-a.s.; an equivalent condition is that $\E\mathbf{d}^{p}_{\rm H}(X,\{\0\})<\infty$ (see, e.g., [Hiai and Umegaki, 1977] \cite{HU}). More generally, a nonempty subset $K\subset\mathds{L}^{p}(D\times\Omega,\mathcal{G},M\times\PP;E)$, for $\mathcal{G}$ a sub-$\upsigma$-algebra of $\mathcal{B}([0,T])\otimes\mathcal{B}(\mathds{R}^{d})\otimes\mathcal{F}$, is called $p$-integrably bounded if there exists $h\in\mathds{L}^{p}(D\times\Omega,\mathcal{G},M\times\PP;\mathds{R}_{+})$ such that $\|f\|\leq h$, $M\times\PP$-a.e., for every $f\in K$. The expectation of the set-valued random variable $X$ conditional on $\mathcal{G}\subset\mathcal{F}$ is defined to be the $\mathcal{G}$-measurable random variable $Z$ such that $S^{p}_{\mathcal{G}}(Z)=\cl\{\E(f|\mathcal{G}):f\in S^{p}_{\mathcal{F}}(X)\}$. The following classical theorem due to [Hiai and Umegaki, 1977, Theorem 3.1] \cite{HU} explains the connection between decomposability and the collection of integrable selectors.

\begin{theorem}\label{thm:1}
Let $K\subset\mathds{L}^{p}(\Omega,\mathcal{G},\PP;E)$ be a nonempty closed subset. Then there exists a $\mathcal{G}$-measurable random variable $X:\Omega\mapsto\mathrm{Cl}(E)$ such that $S^{p}_{\mathcal{G}}(X)=K$ if and only if $K$ is decomposable.
\end{theorem}

The same result applies directly to the (more general) $\mathds{L}^{p}(D\times\Omega,\mathcal{G},M\times\PP;E)$ space, with $\mathcal{G}\subset\mathcal{B}([0,T])\otimes\mathcal{B}(\mathds{R}^{d})\otimes\mathcal{F}$. In fact, for any open or closed balls $B\subset\mathds{L}^{p}(D\times\Omega,\mathcal{G},M\times\PP;E)$ it holds that $\dec B=\mathds{L}^{p}(D\times\Omega,\mathcal{G},M\times\PP;E)$ (see [Fryszkowski, 2004, Proposition 51] \cite{F}); hence it follows that strict subsets of $\mathds{L}^{p}(D\times\Omega,\mathcal{G},M\times\PP;E)$ that are decomposable must have an empty interior. The following lemma, owing to [Michta, 2015, Theorem 2.2] \cite{M5}, is useful for exploring connections between decomposability and integrability later on.

\begin{lemma}\label{lem:2}
Let $K\subset\mathds{L}^{p}(D\times\Omega,\mathcal{G},M\times\PP;E)$ be a nonempty subset. Then the following two assertions are equivalent: \medskip\\
(i) $K$ is $p$-integrably bounded;\\
(ii) $\overline{\dec}_{\mathcal{G}}K$ is a bounded subset of $\mathds{L}^{p}(D\times\Omega,\mathcal{G},M\times\PP;E)$.
\end{lemma}

\medskip

\section{Set-valued convoluted stochastic integrals}\label{sec:3}

In this paper, as mentioned before, we will consider four types of set-valued convoluted stochastic integrals, which are indexed by $q\in\{1,2,3,4\}$ for notational convenience. In what follows the spaces $\mathds{L}^{p}(D_{q}\times\Omega,\mathcal{P}(\mathbb{F}),M_{q}\times\PP;E_{q})$, $q\in\{1,2,3,4\}$, where $D_{1}=D_{2}=[0,T]$, $D_{3}=D_{4}=[0,T]\times(\mathds{R}^{d}\setminus\{\0\})$, $M_{1}=M_{2}=\mathrm{Leb}_{[0,T]}$, $M_{3}=\mathrm{Leb}_{[0,T]}\times\nu\upharpoonright_{\{\|z\|\geq1\}}$, $M_{4}=\mathrm{Leb}_{[0,T]}\times\nu\upharpoonright_{\{0<\|z\|<1\}}$, $E_{1}=E_{3}=E_{4}=\mathds{R}^{d}$, $E_{2}=\mathds{R}^{d\times m}$ will be considered specifically, and $\breve{\mathds{L}}^{p}$ denotes the corresponding subspaces such that conditions (\ref{2.3}) are satisfied, which are closed in $\mathds{L}^{p}$ for $p\geq1$ because $[0,T]$ is compact. Similarly, $\breve{S}^{p}_{\mathcal{P}(\mathbb{F})}$ is used to denote the collection of predictable selectors in the $\breve{\mathds{L}}^{p}$ subspaces. We start with the following definition of the corresponding integral functionals.\footnote{In this paper, when the integrands are set-valued (or multifunctions), for notational consistency we stick to writing nonempty subsets $K$ of $\mathds{L}^{p}$ spaces, thanks to the one-to-one correspondence given by Theorem \ref{thm:1}.}

\begin{definition}\label{def:2}
Let the time points $0\leq t_{0}<t\leq T$ be fixed and let $K$ be a nonempty subset of the $\breve{\mathds{L}}^{p}(D_{q}\times\Omega,\mathcal{P}(\mathbb{F}),M_{q}\times\PP;E_{q})$ space, where the number $p\geq1$ satisfies that $p=2$ for $q\in\{2,4\}$. Then, for a suitable (Volterra-type) kernel $g$ (Definition \ref{def:1}), the ($g$-)convoluted Aumann integral functional, set-valued convoluted It\^{o} integral functional, set-valued convoluted integral functional with respect to the Poisson random measure, and the set-valued convoluted integral functional with respect to the compensated Poisson random measure, of $K$, are defined, in the order of $q$, to be the set-valued mappings
\begin{equation}\label{2.4}
  \mathcal{S}\big(\breve{\mathds{L}}^{p}(D_{q}\times\Omega,\mathcal{P}(\mathbb{F}),M_{q}\times\PP;E_{q})\big)\ni K\mapsto I^{(q,g)}_{t_{0},t}(K):=\big\{I^{(q,g)}_{t_{0},t}(h):h\in K\big\}\in\mathcal{S}(\mathds{L}^{p}(\Omega,\mathscr{F}_{t},\PP;\mathds{R}^{d})),
\end{equation}
where
\begin{equation}\label{2.5}
  I^{(q,g)}_{t_{0},t}(h)=
  \begin{cases}
    \displaystyle \int^{t}_{t_{0}}g(t,s)h(s-)\dd s,\quad&\text{if }q=1,\\
    \displaystyle \int^{t}_{t_{0}}g(t,s)h(s-)\dd W_{s},\quad&\text{if }q=2,\\
    \displaystyle \int^{t}_{t_{0}}\int_{\|z\|\geq1}g(t,s)h(s-,z)N(\dd s,\dd z),\quad&\text{if }q=3,\\
    \displaystyle \int^{t}_{t_{0}}\int_{0<\|z\|<1}g(t,s)h(s-,z)\tilde{N}(\dd s,\dd z),\quad&\text{if }q=4.
  \end{cases}
\end{equation}
\end{definition}

Since the kernel $g$ is deterministic, convolution does not affect the properties of the integral functionals at fixed time, as long as (\ref{2.3}) is satisfied with $g$. The following theorem summarizes some basic properties of the convoluted integral functionals defined in (\ref{2.4}) and (\ref{2.5}).

\begin{theorem}\label{thm:2}
Let $K$ and $H$ be nonempty subsets of the $\breve{\mathds{L}}^{p}(D_{q}\times\Omega,\mathcal{P}(\mathbb{F}),M_{q}\times\PP;E_{q})$ space, $q\in\{1,2,3,4\}$. Then the following assertions hold for any $q\in\{1,2,3,4\}$, any given suitable kernel $g$, and any fixed $0\leq t_{0}\leq t\leq T$: \medskip\\
(i) $I^{(q,g)}_{t_{0},t}(\cl_{\mathds{L}^{p}}K)=\cl_{\mathds{L}^{p}}I^{(q,g)}_{t_{0},t}(K)$; \medskip\\
(ii) $I^{(q,g)}_{t_{0},t}(\overline{\co}_{\mathds{L}^{p}}K)=\overline{\co}_{\mathds{L}^{p}}I^{(q,g)}_{t_{0},t}(K)$; \medskip\\
(iii) $I^{(q,g)}_{t_{0},t}(K+H)=I^{(q,g)}_{t_{0},t}(K)+I^{(q,g)}_{t_{0},t}(H)$, where the sum is in the sense of Minkowski; \medskip\\
(iv) there exists a sequence $\{f_{n}:n\in\mathds{N}_{++}\}\subset K$ such that $\cl_{\mathds{L}^{p}}I^{(q,g)}_{t_{0},t}(K)=\cl_{\mathds{L}^{p}}\big\{I^{(q,g)}_{t_{0},t}(f_{n}):n\in\mathds{N}_{++}\big\}$.
\end{theorem}

\begin{proof}
In the cases $q=1,2$ the results can be readily gleaned from [Kisielewicz, 2020b, Chapter IV \text{Thm.} 4.2.1 and Chapter V \text{Lemma} 5.1.1] \cite{K8}, in light of (\ref{2.3}).

For the rest two cases we only consider $q=4$, while $q=3$ is similar and easier. Note $p=2$ in particular. For assertion (i), it is clear that $I^{(4,g)}_{t_{0},t}(\cl_{\mathds{L}^{2}}K)\subset\cl_{\mathds{L}^{2}}I^{(4,g)}_{t_{0},t}(K)$ because the integral functional is a (continuous) linear isometry. On the other hand, any sequence $\{f_{n}:n\in\mathds{N}_{++}\}\subset K$ such that $\E\big\|I^{4,g}_{t_{0},t}(f_{n})-\eta\big\|^{2}\rightarrow0$ for an arbitrary $\eta\in\cl_{\mathds{L}^{2}}I^{(4,g)}_{t_{0},t}(K)$ satisfies
\begin{equation*}
  \E\big\|I^{(4,g)}_{t_{0},t}(f_{m})-I^{(4,g)}_{t_{0},t}(f_{n})\big\|^{2}=\E\int^{t}_{t_{0}} \int_{0<\|z\|<1}g^{2}(t,s)\|f_{m}(s-,z)-f_{n}(s-,z)\|^{2}\nu(\dd z)\dd s\rightarrow0,\;\text{as }m,n\rightarrow\infty,
\end{equation*}
because $g(t,\cdot)\in\mathds{L}^{2}([0,t);\mathds{R})$. This implies that $\{f_{n}\}$ is a Cauchy sequence in the $\breve{\mathds{L}}^{2}([0,T]\times(\mathds{R}^{d}\setminus\{\0\})\times\Omega,\mathcal{P}(\mathbb{F}), \mathrm{Leb}_{[0,T]}\times\nu\upharpoonright_{\{0<\|z\|<1\}}\times\PP;\mathds{R}^{d})$ space, which by completeness gives that there exists $f\in\breve{\mathds{L}}^{2}([0,T]\times(\mathds{R}^{d}\setminus\{\0\})\times\Omega,\mathcal{P}(\mathbb{F}),\mathrm{Leb}_{[0,T]}\times \nu\upharpoonright_{\{0<\|z\|<1\}}\times\PP;\mathds{R}^{d})$ such that
\begin{equation*}
  \E\int^{t}_{t_{0}}\int_{0<\|z\|<1}g^{2}(t,s)\|f(s-,z)-f_{n}(s-,z)\|^{2}\nu(\dd z)\dd s\rightarrow0,\quad\text{as }n\rightarrow\infty,
\end{equation*}
so that $f\in\cl_{\mathds{L}^{2}}K$ and $\E\big\|I^{(4,g)}_{t_{0},t}(f_{m})-I^{(4,g)}_{t_{0},t}(f)\big\|^{2}$ tends to 0. Therefore, $I^{(4,g)}_{t_{0},t}(f)\in I^{(4,g)}_{t_{0},t}(\cl_{\mathds{L}^{2}}K)$, and since $\eta$ is arbitrary it follows that $I^{(4,g)}_{t_{0},t}(\cl_{\mathds{L}^{2}}K)\supset\cl_{\mathds{L}^{2}}I^{(4,g)}_{t_{0},t}(K)$.

Then, assertions (ii) and (iii) easily follow by the linearity of the functional $I^{(4,g)}_{t_{0},t}$ on the basis of assertion (i).

For assertion (iv), since the induced space of $(\Omega,\mathcal{F},\PP)$ is separable it is understood that the $\mathds{L}^{2}([0,T]\times(\mathds{R}\setminus\{\0\})\times\Omega,\mathcal{P}(\mathbb{F}),\mathrm{Leb}_{[0,T]}\times\nu\upharpoonright_{\{0<\|z\|<1\}} \times\PP;\mathds{R}^{d})$ (Banach) space is separable. Since the corresponding $\breve{\mathds{L}}^{2}$ subspace is closed, it is also separable. Then, we observe that there exists a sequence $\{f_{n}:n\in\mathds{N}_{++}\}\subset K$ such that
\begin{equation*}
  K=K\cap\cl_{\mathds{L}^{2}}\{f_{n}:n\in\mathds{N}_{++}\}.
\end{equation*}
Taking closure on both sides, together with the fact that $\cl_{\mathds{L}^{2}}\{f_{n}:n\in\mathds{N}_{++}\}\subset\cl_{\mathds{L}^{2}}K$, establishes the equivalence $\cl_{\mathds{L}^{2}}K=\cl_{\mathds{L}^{2}}\{f_{n}:n\in\mathds{N}_{++}\}$, which, along with assertion (i), proves the result.
\end{proof}

We remark that working with the $\breve{\mathds{L}}^{p}$ subspaces is essential for convoluted stochastic integrals. This allows one to choose any suitable Volterra-type kernel without having to deal with reproduced kernel Hilbert spaces (see, e.g., [Hult, 2003] \cite{H}) that are not necessarily $\mathds{L}^{p}$ spaces.

\medskip

\section{Decomposability and integrability}\label{sec:4}

In this section we give the formal definition of the four types of set-valued convoluted stochastic integrals, after studying the decomposability of the corresponding functionals. Then, we present results on the integrable boundedness properties of these integrals, on the basis of Lemma \ref{lem:2}. The next lemma shows the general non-decomposability of the integral functionals.

\begin{lemma}\label{lem:3}
Let $0\leq t_{0}<t\leq T$ be fixed and suppose that $K\subset\breve{\mathds{L}}^{p}(D_{q}\times\Omega,\mathcal{P}(\mathbb{F}),M_{q}\times\PP;E_{q})$, for $q\in\{1,2,3,4\}$, is nonempty and decomposable. Then, for a given suitable kernel $g$ and each $q$, $I^{(q,g)}_{t_{0},t}(K)$ is decomposable (with respect to $\mathscr{F}_{t}$) if and only if $\card K=1$.
\end{lemma}

\begin{proof}
First, by non-emptiness and the definition of decomposability, it is easily justifiable that $\card K\in\{1,2^{\mathfrak{c}}\}$, $\mathfrak{c}$ being the continuum cardinality. Clearly, if $\card K=1$ then all four types of integrals under consideration are single-valued, so that $\dec I^{(q,g)}_{t_{0},t}(K)=I^{(q,g)}_{t_{0},t}(K)$ for each $q$, which takes care of the sufficiency parts.

For $q=1,3$, if $\card K>1$, we can choose $A(t)\in\mathscr{F}_{t-}\setminus\upsigma\big(\bigcup_{s\in[t_{0},t_{1}]}\mathscr{F}_{s}\big)$ for some $t_{1}\in[t_{0},t)$, and two $\mathbb{F}$-predictable processes $\phi,\gamma\in K$ with $\phi\not\equiv\0$. Then as $I^{(q,g)}_{t_{0},t}$ is a linear isometry one has for $q=1$ that
\begin{equation*}
  I^{(1,g)}_{t_{0},t}\big(\phi\mathds{1}_{[t_{0},t]\times A(t)}+\gamma\mathds{1}_{[t_{0},t]\times A^{\complement}(t)}\big)=\mathds{1}_{A(t)}I^{(1,g)}_{t_{0},t}(\phi)+\mathds{1}_{A^{\complement}(t)}I^{(1,g)}_{t_{0},t}(\gamma).
\end{equation*}
The right-hand side cannot belong to $I^{(1,g)}_{t_{0},t}(K)$ since the process $\phi\mathds{1}_{[t_{0},t]\times A(t)}+\gamma\mathds{1}_{[t_{0},t]\times A^{\complement}(t)}$ defined over $[t_{0},t]$ fails to be $\mathbb{F}$-predictable when restricted to the subinterval $[t_{0},t_{1}]$. In the same spirit, for $q=3$ we establish that
\begin{equation*}
  I^{(3,g)}_{t_{0},t}(\phi\mathds{1}_{[t_{0},t]\times A(t)}+\gamma\mathds{1}_{[t_{0},t]\times A^{\complement}(t)})=\mathds{1}_{A(t)}I^{(3,g)}_{t_{0},t}(\phi)+\mathds{1}_{A^{\complement}(t)}I^{(3,g)}_{t_{0},t}(\gamma),
\end{equation*}
which does not belong to $I^{(3,g)}_{t_{0},t}(K)$ either. Therefore, $I^{(q,g)}_{t_{0},t}(K)$ cannot be decomposable for $q=1,3$ unless $\card K=1$.

For $q=2,4$, a different argument is used, exploiting the centeredness property. Suppose $\dec_{\mathscr{F}_{t}}I^{(q,g)}_{t_{0},t}(K)=I^{(q,g)}_{t_{0},t}(K)$, so that for any $A(t)\in\mathscr{F}_{t-}$ and any $f,l\in I^{(q,g)}_{t_{0},t}(K)$, $\mathds{1}_{A(t)}f+\mathds{1}_{A^{\complement}(t)}l\in I^{(q,g)}_{t_{0},t}(K)$. On the other hand, since $g$ is square-integrable, it is clear that $\int^{t}_{t_{0}}g(t,s)f(s-)\dd W_{s}$ and $\int^{t}_{t_{0}}\int_{0<\|z\|<1}g(t,s)f(s-,z)\tilde{N}(\dd s,\dd z)$ are both square-integrable and centered random variables, and hence taking expectation yields
\begin{align*}
  \E\big(\mathds{1}_{A(t)}f+\mathds{1}_{A^{\complement}(t)}l\big)&=\E\big(\mathds{1}_{A(t)}(f-l)+l\big) \\
  &\in\E I^{(q,g)}_{t_{0},t}(K):=\cl\big\{\E\ell:\ell\in I^{(q,g)}_{t_{0},t}(K)\big\}=\cl\{\0\}=\{\0\}.
\end{align*}
Taking $f=I^{(q,g)}_{t_{0},t}(\phi)$ and $l=I^{(q,g)}_{t_{0},t}(\gamma)$ for $\phi,\gamma\in K$, then $I^{(q,g)}_{t_{0},t}(\phi-\gamma)=I^{(q,g)}_{t_{0},t}(\phi)-I^{(q,g)}_{t_{0},t}(\gamma)=\0$ by the arbitrariness of $A(t)\in\mathscr{F}_{t}$. It follows by the L\'{e}vy--It\^{o} isometry that $\E\big\|I^{(2,g)}_{t_{0},t}(\phi-\gamma)\big\|^{2}=\E\int^{t}_{t_{0}}g^{2}(t,s)\|\phi(s-)-\gamma(s-)\|^{2}\dd s=0$ and $\E\big\|I^{(4,g)}_{t_{0},t}(\phi-\gamma)\big\|^{2}=\E\int^{t}_{t_{0}}\int_{0<\|z\|<1}g^{2}(t,s)\|\phi(s-,z)-\gamma(s-,z)\|^{2}$ $\nu(\dd z)\dd s=0$, implying that $\phi=\gamma$, and so $\card K=1$.
\end{proof}

Lemma \ref{lem:3} basically means that the decomposability of all four types of set-valued integral functionals is only guaranteed when $K$ is a singleton. This is to be expected based on the theory of set-valued stochastic integrals in the absence of convolution and jumps (mentioning [Kisielewicz, 2012, Theorem 2.1] \cite{K5} and [Kisielewicz, 2020b, Chapter V Lemma 5.1.1] \cite{K8}). This result also strengthens the following definition.

\begin{definition}\label{def:3}
In the setting of Definition \ref{def:2}, for a general subset $K\in\breve{\mathds{L}}^{p}(D_{q}\times\Omega,\mathcal{P}(\mathbb{F}),M_{q}\times\PP;E_{q})$ and a suitable kernel $g$, the ($g$-)convoluted Aumann stochastic integral, set-valued convoluted It\^{o} integral, set-valued convoluted integral with respect to the Poisson random measure, and the set-valued convoluted integral with respect to the compensated Poisson random measure, of $K$, are defined, in the order of $q$, as the $\mathscr{F}_{t}$-measurable random variable $X_{q}$ such that
\begin{equation*}
  S^{p}_{\mathscr{F}_{t}}(X_{q})=\overline{\dec}_{\mathscr{F}_{t}}I^{(q,g)}_{t_{0},t}(K),\quad q\in\{1,2,3,4\}.
\end{equation*}
\end{definition}

In Definition \ref{def:3}, the random variable $X_{q}$ will be respectively expressed as $\int^{t}_{t_{0}}g(t,s)K(s-)\dd s$, $\int^{t}_{t_{0}}g(t,s)K(s-)\dd W_{s}$, $\int^{t}_{t_{0}}\int_{\|z\|\geq1}g(t,s)K(s-,z)N(\dd s,\dd z)$, and $\int^{t}_{t_{0}}\int_{0<\|z\|<1}g(t,s)K(s-,z)\tilde{N}(\dd s,\dd z)$, in the order of $q$. Taking the left time limit in $K$ is to be understood as doing so for each element.

We remark that, with square-integrable kernels and the conditions in (\ref{2.3}), the integrands of the set-valued convoluted integrals are guaranteed to be in the corresponding $\mathds{L}^{p}(D_{q}\times\Omega,\mathcal{P}(\mathbb{F}),M_{q}\times\PP;E_{q})$ spaces. With this in mind, we can state the following theorem, by consulting [Kisielewicz and Michta, 2017, Theorem 2.8] \cite{KM} for $q=1$, [Michta, 2015, Theorem 3.15] \cite{M5} and [Kisielewicz, 2020a, Theorem 11] \cite{K7} for $q=2$, and [Zhang et al., 2013, Theorem 3.3] \cite{ZMO} for $q=3$. Note that the L\'{e}vy measure is restricted to $\{z\in\mathds{R}^{d}:\|z\|\geq1\}$ in this setting, and is essentially a finite measure.

\begin{theorem}\label{thm:3}
Let $0\leq t_{0}<t\leq T$ be fixed and let $K$ be a nonempty closed bounded decomposable subset of the $\breve{\mathds{L}}^{2}(D_{q}\times\Omega,\mathcal{P}(\mathbb{F}),M_{q}\times\PP;E_{q})$ space, for $q\in\{1,2,3\}$. Then, for a suitable kernel $g$, $\overline{\dec}_{\mathscr{F}_{t}}I^{(q,g)}_{t_{0},t}(K)$ is a bounded subset of the $\mathds{L}^{p}(\Omega,\mathscr{F}_{t},\PP;\mathds{R}^{d})$ space if, and only if, either $q=1,3$ or $\card K=1$.
\end{theorem}

In light of Lemma \ref{lem:2}, Theorem \ref{thm:3} basically concludes that the convoluted Aumann stochastic integral and set-valued convoluted stochastic integral with respect to the Poisson random measure are both integrably bounded whenever the integrand is integrable (as a subset of the $\mathds{L}^{p}$ space). However, the set-valued convoluted It\^{o} integral is not square-integrably bounded\footnote{If one is only interested in integrable unboundedness (of order 1), the result remains negative, which is easily seen from the proof of [Michta, 2015, Theorem 3.5] \cite{M5}.} unless the integrand is a singleton.

For the integral with respect to the compensated measure ($q=4$), we present a novel result which implies its potential integrable unboundedness. The theorem is enough for the conclusions of our study but is unfortunately provable based on specific symmetry assumptions because of insufficient theory on the related inequalities for more general L\'{e}vy processes and is hence non-exhaustive. Nevertheless, we indeed conjecture that a more general result would hold for arbitrary $K\subset\mathds{L}^{2}(D_{4}\times\Omega,\mathcal{P}(\mathbb{F}),M_{4}\times\PP;E_{4})$, like proved in [Kisielewicz, 2020a] \cite{K7} for It\^{o}-type integrals, which is to be addressed in a future work.

Beforehand, we recall that a symmetric $\alpha$-stable random vector $Y$ with values in $\mathds{R}^{n}$ ($n\in\mathds{N}_{++}$) is one with characteristic function taking the form
\begin{equation}\label{4.1}
  \E e^{\ii u\cdot Y}=\exp\bigg(-\int_{S_{n}}|u\cdot x|^{\alpha}\varsigma(\dd x)\bigg),\quad u\in\mathds{R}^{n},
\end{equation}
for some finite measure $\varsigma$ on the $n$-dimensional unit sphere $S_{n}$, referred to as the spectral measure (see, e.g., [Samorodnitsky and Taqqu, 1993, Equation (1.3)] \cite{ST1}). The following lemma is a restatement of [Marcus, 1984, Theorem 2.2] \cite{M3}, which provides a lower bound on the maximal values of stable random vectors.

\begin{lemma}\label{lem:4}
Let $Y$ be a symmetric $\alpha$-stable ($1<\alpha<2$) random vector in $\mathds{L}^{1}(\Omega;\mathds{R}^{n})$, $n\geq2$. Then, for $r\in(0,\alpha)$ and $1/\alpha+1/\alpha'=1$, it holds that
\begin{equation*}
  \E\max_{j\in\mathds{N}\cap[1,n]}|Y_{j}|\geq c_{r,\alpha}\inf_{j<j'}(\E(Y_{j}-Y_{j'})^{r})^{1/r}\log^{1/\alpha'}n,
\end{equation*}
where $c_{r,\alpha}$ is a positive constant depending only on $r$ and $\alpha$.
\end{lemma}

\begin{theorem}\label{thm:4}
Let $0\leq t_{0}<t\leq T$ be fixed, $K$ be a nonempty closed bounded decomposable subset of the $\breve{\mathds{L}}^{2}\big([0,T]\times(\mathds{R}^{d}\setminus\{\0\})\times\Omega,\mathcal{P}(\mathbb{F}), \mathrm{Leb}_{[0,T]}\times\nu\upharpoonright_{\{0<\|z\|<1\}}\times\PP;\mathds{R}^{d}\big)$ space, and $g$ be a suitable kernel. If either $\int_{0<\|z\|<1}\|z\|\nu(\dd z)<\infty$ or $\card K=1$, then $\overline{\dec}_{\mathscr{F}_{t}}I^{(4,g)}_{t_{0},t}(K)$ is a bounded subset of the $\mathds{L}^{2}(\Omega,\mathscr{F}_{t},\PP;\mathds{R}^{d})$ space. If $\int_{0<\|z\|<1}\|z\|\nu(\dd z)=\infty$, assume further that there exists a sequence $\{b_{j}z:j\in\mathds{N}_{++}\}\subset K$ ($b_{j}$ valued in $\mathds{R}^{d\times d}$, $\forall j$) such that, for some $i^{\ast}\in\mathds{N}\cap[1,d]$ and any $i\in\mathds{N}\cap[1,d]$, $\{b_{j,i^{\ast},i}\}\subset\mathds{L}^{\infty}([0,T];\mathds{R})$, $b_{j,i^{\ast},i}$'s are nonzero and distinct a.e., and for some $\epsilon>0$ and $r\in[1,2)$,
\begin{equation}\label{4.2}
  \inf_{j<j'}\int^{t}_{t_{0}}\min_{1\leq i\leq d}|b_{j,i^{\ast},i}(s)-b_{j',i^{\ast},i}(s)|^{r}\dd s=\epsilon;
\end{equation}
then $\overline{\dec}_{\mathscr{F}_{t}}I^{(4,g)}_{t_{0},t}(K)$ is not necessarily a bounded subset of the $\mathds{L}^{2}\big(\Omega,\mathscr{F}_{t},\PP;\mathds{R}^{d}\big)$ space.
\end{theorem}

\begin{proof}
The proof is done in three steps, concerning three situations respectively.

\textbf{Step 1.} Suppose $\card K=1$. Then $\int^{t}_{t_{0}}\int_{0<\|z\|<1}g(t,s)K(s-,z)\tilde{N}(\dd s,\dd z)$ is single-valued, and so $\overline{\dec}_{\mathscr{F}_{t}}I^{(4,g)}_{t_{0},t}(K)$ is a singleton and hence a bounded subset.

\textbf{Step 2.} Suppose $\card K>1$ but $\int_{0<\|z\|<1}\|z\|\nu(\dd z)=V_{1}<\infty$. In this case, we can rewrite
\begin{equation}\label{4.3}
  I^{(4,g)}_{t_{0},t}(K)=\bigg\{\int^{t}_{t_{0}}\int_{0<\|z\|<1}g(t,s)h(s-,z)N(\dd s,\dd z)-\int^{t}_{t_{0}}\int_{0<\|z\|<1}g(t,s)h(s-,z)\nu(\dd z)\dd s: h\in K\bigg\},
\end{equation}
where both (single-valued) integral terms on the right-hand side have finite-variation integrators and are well-defined as Lebesgue--Stieltjes integrals. Then, we apply the result in [Hiai and Umegaki, 1977, Theorem 2.2] \cite{HU} to interchange the expectation with supremum and obtain based on Definition \ref{def:3},
\begin{align*}
  \E\mathbf{d}^{2}_{\rm H}\bigg(\int^{t}_{t_{0}}\int_{0<\|z\|<1}g(t,s)K(s-,z)\tilde{N}(\dd s,\dd z),\{\0\}\bigg)&=\sup\big\{\E\|\eta\|^{2}:\eta\in\overline{\dec}_{\mathscr{F}_{t}}I^{(4,g)}_{t_{0},t}(K)\big\} \\ &=\sup\big\{\E\|\eta\|^{2}:\eta\in\dec_{\mathscr{F}_{t}}I^{(4,g)}_{t_{0},t}(K)\big\}.
\end{align*}
To argue the finiteness of the last supremum, note that by Definition \ref{def:3}, any $d$-dimensional random vector $\eta\in\dec_{\mathscr{F}_{t}}I^{(4,g)}_{t_{0},t}(K)$ admits the representation
\begin{equation*}
  \eta=\sum^{n}_{j=1}\mathds{1}_{A_{j}(t)}\int^{t}_{t_{0}}\int_{0<\|z\|<1}g(t,s)h_{j}(s-,z)\tilde{N}(\dd s,\dd z),
\end{equation*}
for some $\mathscr{F}_{t}$-measurable partition $\{A_{j}(t):j\in\{1,\dots,n\}\}$ of $\Omega$, some (finite) sequence $\{h_{j}:j\in\{1,\dots,n\}\}\subset K$, $n\geq1$. For any such $\eta$, assuming further that $\nu(0<\|z\|<1)=V_{0}<\infty$, we then have that
\begin{align}\label{4.4}
  \E\|\eta\|^{2}&=\E\Bigg\|\sum^{n}_{j=1}\mathds{1}_{A_{j}(t)}\int^{t}_{t_{0}}\int_{0<\|z\|<1}g(t,s)h_{j}(s-,z)\tilde{N}(\dd s,\dd z)\Bigg\|^{2} \nonumber\\
  &=\sum^{n}_{j=1}\E\bigg(\mathds{1}_{A_{j}(t)}\bigg\|\int^{t}_{t_{0}}\int_{0<\|z\|<1}g(t,s)h_{j}(s-,z)\tilde{N}(\dd s,\dd z)\bigg\|^{2}\bigg) \nonumber\\
  &\leq2V_{0}(t-t_{0})\sum^{n}_{j=1}\bigg(\E\bigg(\mathds{1}_{A_{j}(t)}\int^{t}_{t_{0}}\int_{0<\|z\|<1}\|g(t,s)h_{j}(s-,z)\|^{2}N(\dd s,\dd z)\bigg) \nonumber\\
  &\qquad+\E\bigg(\mathds{1}_{A_{j}(t)}\int^{t}_{t_{0}}\int_{0<\|z\|<1}\|g(t,s)h_{j}(s-,z)\|^{2}\nu(\dd z)\dd s\bigg)\bigg) \nonumber\\
  &\leq4V_{0}(t-t_{0})\E\int^{t}_{t_{0}}\int_{0<\|z\|<1}\mathbf{d}^{2}_{\rm H}(g(t,s)K(s-,z),\{\0\})\nu(\dd z)\dd s \nonumber\\
  &<\infty,
\end{align}
where the first equality follows from the disjointness of $A_{j}$'s (see also [Kisielewicz and Michta, 2017, Lemma 2.5 and Lemma 2.6] \cite{KM}) and the first inequality is a result of (8) and H\"{o}lder's inequality, and finiteness follows from the square-integrable boundedness of $gK$ (recall (\ref{2.3})).

Similarly, if $\nu(\{0<\|z\|<1\})=\infty$, following the first equality in (\ref{4.4}) we have instead
\begin{equation}\label{4.5}
  \E\|\eta\|^{2}\leq4\sqrt{V_{1}(t-t_{0})}\E\bigg(\int^{t}_{t_{0}}\int_{0<\|z\|<1}\mathbf{d}^{2}_{\rm H}\bigg(\frac{g(t,s)K(s-,z)}{\|z\|},\{\0\}\bigg)\|z\|\nu(\dd z)\dd s\bigg)^{1/2}<\infty,
\end{equation}
where the last expectation is also finite because $g(t,s)K(s-,z)/\|z\|$ is square-integrably bounded with respect to the (finite) measure $\mathrm{Leb}_{[t_{0},t)}\times\int_{\cdot}\|z\|\nu(\dd z)$. The second last inequalities in both (\ref{4.4}) and (\ref{4.5}) ensure that $\E\mathbf{d}^{2}_{\rm H}\big(\int^{t}_{t_{0}}\int_{0<\|z\|<1}g(t,s)K(s-,z)\tilde{N}(\dd s,\dd z),\{\0\}\big)<\infty$, which with Lemma \ref{lem:2} implies that $\overline{\dec}_{\mathscr{F}_{t}}I^{(4,g)}_{t_{0},t}(K)$ is a bounded subset.

\textbf{Step 3.} Suppose that $\card K>1$ and $\int_{0<\|z\|<1}\|z\|\nu(\dd z)=\infty$, so that the separation in (\ref{4.3}) is invalid. We observe that for every sequence $\{X_{j;t}:j\in\mathds{N}\cap[1,n]\}$, $n\geq2$, of $d$-dimensional random vectors within $\overline{\dec}_{\mathscr{F}_{t}}I^{(4,g)}_{t_{0},t}(K)$ it holds that
\begin{equation}\label{4.6}
  \E\mathbf{d}^{p}_{\rm H}\bigg(\int^{t}_{t_{0}}\int_{0<\|z\|<1}g(t,s)K(s-,z)\tilde{N}(\dd s,\dd z),\{\0\}\bigg)\geq\E\Big(\max_{1\leq j\leq n}\|X_{j;t}\|\Big)^{p},\quad p\in[1,2].
\end{equation}
By the stated assumptions, we can assign $\int^{t}_{t_{0}}\int_{0<\|z\|<1}g(t,s)b_{j}(s-)z\tilde{N}(\dd s,\dd z):=X_{j;t}$ for each $j$, with $X_{j,i;t}=\int^{t}_{t_{0}}\int_{0<\|z\|<1}[g(t,s)b_{j}(s-)z]_{i}\tilde{N}(\dd s,\dd z)$ for any $i\in\mathds{N}\cap[1,d]$.

It is clear that each $\int^{t}_{t_{0}}\int_{0<\|z\|<1}[g(t,s)b_{j}(s-)z]_{i}\tilde{N}(\dd s,\dd z)$ has an infinitely divisible distribution and is associated with some (generally time-inhomogeneous) L\'{e}vy measure -- call it $\bar{\nu}_{j,i;t}$. Based on condition (ii) of Definition \ref{def:1} and the fact that $b_{j}(s)z\neq\0$ for $\mathrm{Leb}_{[t_{0},t)}\times\nu\upharpoonright_{\{0<\|z\|<1\}}$-\text{a.e.} $(s,z)\in[t_{0},t)\times(\mathds{R}^{d}\setminus\{\0\})$ and $\int^{t}_{t_{0}}\int_{0<\|z\|<1}g^{2}(t,s)\|b_{j}(s-)z\|^{2}\nu(\dd z)\dd s<\infty$, the L\'{e}vy--Khintchine formula gives that for every $i,j$ and some $c_{i}>0$,
\begin{align}\label{4.7}
  \log\E e^{\ii uX_{j,i;t}}&=\int^{t}_{t_{0}}\int_{0<\|z\|<1}\big(e^{\ii u[g(t,s)b_{j}(s)z]_{i}}-1-\ii u[g(t,s)b_{j}(s)z]_{i}\big)\nu(\dd z)\dd s \nonumber\\
  &=\int_{0<|z_{i}|<c_{i}}\big(e^{\ii uz_{i}}-1-\ii uz_{i}\big)\bar{\nu}_{j,i;t}(\dd z_{i}),\quad u\in\mathds{R},
\end{align}
from where it is clear that $\bar{\nu}_{j,i;t}(\{0<|z_{i}|<c_{i}\})=\infty$ if and only if $\nu(\{0<\|z\|<1\})=\infty$, and $\int_{0<|z_{i}|<c_{i}}|z_{i}|\bar{\nu}_{j,i;t}(\dd z_{i})=\infty$ if and only if $\int_{0<\|z\|<1}\|z\|\nu(\dd z)=\infty$.

Since it is known that $\overline{\dec}_{\mathscr{F}_{t}}I^{(3,g)}_{t_{0},t}(K)$ is bounded (Theorem \ref{thm:3}), it is equivalent to consider the unboundedness of $\overline{\dec}_{\mathscr{F}_{t}}\big(I^{(3,g)}_{t_{0},t}(K)+I^{(4,g)}_{t_{0},t}(K)\big)$; indeed, for any $p\in[1,2]$, following Definition \ref{def:3}, let $\int^{t}_{t_{0}}\int_{\|z\|\geq1}g(t,s)K(s-,z)N(\dd s,\dd z)+\int^{t}_{t_{0}}\int_{0<\|z\|<1}g(t,s)K(s-,z)\tilde{N}(\dd s,\dd z)$ denote the random variable whose collection of ($\mathscr{F}_{t}$-measurable) $p$-integral selectors is $\overline{\dec}_{\mathscr{F}_{t}}\big(I^{(3,g)}_{t_{0},t}(K)+I^{(4,g)}_{t_{0},t}(K)\big)$, and then by consulting again [Hiai and Umegaki, 1977, Theorem 2.2] \cite{HU} we have that
\begin{align}\label{4.8}
  &\quad\E\mathbf{d}^{p}_{\rm H}\bigg(\int^{t}_{t_{0}}\int_{\|z\|\geq1}g(t,s)K(s-,z)N(\dd s,\dd z)+\int^{t}_{t_{0}}\int_{0<\|z\|<1}g(t,s)K(s-,z)\tilde{N}(\dd s,\dd z),\{\0\}\bigg) \nonumber\\
  &=\sup\big\{\E\|\eta\|^{p}:\eta\in\overline{\dec}_{\mathscr{F}_{t}}\big(I^{(3,g)}_{t_{0},t}(K)+I^{(4,g)}_{t_{0},t}(K)\big)\big\} \nonumber\\
  &\leq\sup\big\{\E\|\eta\|^{p}:\eta\in\overline{\dec}_{\mathscr{F}_{t}}I^{(3,g)}_{t_{0},t}(K) +\overline{\dec}_{\mathscr{F}_{t}}I^{(4,g)}_{t_{0},t}(K)\big\} \nonumber\\
  &\leq C_{p}\big(\sup\big\{\E\|\eta\|^{p}:\eta\in\overline{\dec}_{\mathscr{F}_{t}}I^{(3,g)}_{t_{0},t}(K)\big\} +\sup\big\{\E\|\theta\|^{p}:\theta\in\overline{\dec}_{\mathscr{F}_{t}}I^{(4,g)}_{t_{0},t}(K)\big\}\big) \nonumber\\
  &=C_{p}\bigg(\E\mathbf{d}^{p}_{\rm H}\bigg(\int^{t}_{t_{0}}\int_{\|z\|\geq1}g(t,s)K(s-,z)N(\dd s,\dd z),\{\0\}\bigg) \nonumber\\
  &\qquad\quad+\E\mathbf{d}^{p}_{\rm H}\bigg(\int^{t}_{t_{0}}\int_{0<\|z\|<1}g(t,s)K(s-,z)\tilde{N}(\dd s,\dd z),\{\0\}\bigg)\bigg),
\end{align}
for a constant $C_{p}>0$ depending only on $p$, where the first inequality follows from the fundamental properties of the decomposable hull; see, e.g., [Kisielewicz, 2020b, Chapter III Theorem 3.3.4] \cite{K8}.

To that end, let us take $\nu$ to be the L\'{e}vy measure of a $d$-dimensional symmetric $\alpha$-stable random vector ($\alpha\in(1,2)$), which is further assumed to have \text{i.i.d.} components for convenience. Its L\'{e}vy measure is hence supported on the (punctured) union of axes $\bigcup^{d}_{i=1}\{[0,\dots,0,z_{i},0,\dots,0]:z_{i}\in\mathds{R}\}\setminus\{\0\}$, and let $\nu_{1}$ denote the L\'{e}vy measure (on $\mathds{R}\setminus\{0\}$) of the first component. Now, consider $i^{\ast}\in\mathds{N}\cap[1,d]$ fixed as mentioned in the theorem. Following (\ref{4.7}), since $gb_{j}$ is for each $j$ a deterministic function, by the fundamental property of stable distributions the $n$-dimensional random vector
\begin{equation*}
  Y_{i^{\ast};t}:=\bigg[\int^{t}_{t_{0}}\int_{\|z\|\geq1}[g(t,s)b_{j}(s-)z]_{i^{\ast}}N(\dd s,\dd z)+\int^{t}_{t_{0}}\int_{0<\|z\|<1}[g(t,s)b_{j}(s-)z]_{i^{\ast}}\tilde{N}(\dd s,\dd z):j\in\mathds{N}\cap[1,n]\bigg]
\end{equation*}
also has a symmetric $\alpha$-stable distribution (with the same $\alpha$). Indeed, the associated L\'{e}vy measure can be expressed as
\begin{align}\label{4.9}
  \bar{\nu}_{i^{\ast};t}(\dd y)&=\sum^{d}_{i=1}\int^{t}_{t_{0}}\nu_{1}\bigg(\frac{\dd y_{1}}{g(t,s)b_{1,i^{\ast},i}(s)}:\; \frac{y_{j}}{b_{j,i^{\ast},i}(s)}=\frac{y_{j'}}{b_{j',i^{\ast},i}(s)},\;\forall j<j'\bigg)\dd s \nonumber\\
  &=\sum^{d}_{i=1}\int^{t}_{t_{0}}C_{\alpha}\big|g(t,s)b_{1,i^{\ast},i}(s)\big|^{\alpha} \mathds{1}_{\{y_{j}/b_{j,i^{\ast},i}(s)=y_{j'}/b_{j',i^{\ast},i}(s),\;\forall j<j'\}}\dd s\frac{\dd y_{1}}{|y_{1}|^{\alpha+1}},\quad y\in\mathds{R}^{n}\setminus\{\0\},
\end{align}
where $C_{\alpha}$ is a positive constant depending only on $\alpha$. Recall that $b_{j,i^{\ast},i}$'s are nonzero a.e.

Let $\varsigma_{1}$ be the spectral measure (on the unit circle $S_{1}$) corresponding to $\nu_{1}$, as in (\ref{4.1}). In the following, $C_{r,\alpha},C'_{r,\alpha},C''_{r,\alpha}>0$ are constants depending only on $r$ and $\alpha$. By consulting results on the representations of stable processes (see [Marcus and Pisier, 1984, Section 1] \cite{MP}), we have that for any $r\in[1,\alpha)$,
\begin{align*}
  \E|Y_{j,i^{\ast};t}-Y_{j',i^{\ast};t}|^{r}&=C_{r,\alpha}\bigg(\int^{t}_{t_{0}}\int_{S_{1}}|g(t,s)|^{\alpha}\Bigg|\sum^{d}_{i=1}(b_{j,i^{\ast},i}(s) -b_{j',i^{\ast},i}(s))x\bigg|^{\alpha}\varsigma_{1}(\dd x)\dd s\bigg)^{r/\alpha} \nonumber\\
  &=C'_{r,\alpha}\bigg(\int^{t}_{t_{0}}|g(t,s)|^{\alpha}\Bigg|\sum^{d}_{i=1}(b_{j,i^{\ast},i}(s)-b_{j',i^{\ast},i}(s))\bigg|^{\alpha}\dd s\bigg)^{r/\alpha} \nonumber\\
  &\geq C'_{r,\alpha}(t-t_{0})^{r/\alpha-1}\int^{t}_{t_{0}}|g(t,s)|^{r}\Bigg|\sum^{d}_{i=1}(b_{j,i^{\ast},i}(s)-b_{j',i^{\ast},i}(s))\Bigg|^{r}\dd s,
\end{align*}
where the last two lines use the finiteness of $\varsigma_{1}$ and Jensen's inequality, respectively. Given that condition (\ref{4.2}) is in force, we further have that
\begin{align*}
  \E|Y_{j,i^{\ast};t}-Y_{j',i^{\ast};t}|^{r}&\geq C'_{r,\alpha}(t-t_{0})^{r/\alpha-1}d^{r}\int^{t}_{t_{0}}|g(t,s)|^{r}\min_{1\leq i\leq d}|b_{j,i^{\ast},i}(s)-b_{j',i^{\ast},i}(s)|^{r} \dd s \\
  &\geq C'_{r,\alpha}(t-t_{0})^{r/\alpha-1}d^{r}\epsilon>0.
\end{align*}
Then, by applying Lemma \ref{lem:4} it follows that
\begin{equation}\label{4.10}
  \Big(\E\max_{1\leq j\leq n}|Y_{j,i^{\ast};t}|\Big)^{r}\geq C''_{r,\alpha}(t-t_{0})^{r/\alpha-1}d^{r}\epsilon\log^{r(\alpha-1)/\alpha}n\rightarrow\infty,\quad\text{as }n\rightarrow\infty.
\end{equation}

Notably, with the stable random vector $Y_{i^{\ast};t}$, for $K\supset\{b_{j}z:j\in\mathds{N}\cap[1,n]\}$ it is only ensured that $K\subset\breve{\mathds{L}}^{r}\big([0,T]\times(\mathds{R}^{d}\setminus\{\0\})\times\Omega,\mathcal{P}(\mathbb{F}), \mathrm{Leb}_{[0,T]}\times\nu\times\PP;\mathds{R}^{d}\big)$ ($r<2$), while it is always true that $K\subset\breve{\mathds{L}}^{2}\big([0,T]\times(\mathds{R}^{d}\setminus\{\0\})\times\Omega,\mathcal{P}(\mathbb{F}), \mathrm{Leb}_{[0,T]}\times\nu\upharpoonright_{\{0<\|z\|<1\}}\times\PP;\mathds{R}^{d}\big)$, as assumed. Hence, to proceed we can set $p=r$ in (\ref{4.6}) and (\ref{4.8}) and replace the symbol ``$X$'' in (\ref{4.6}) with ``$Y_{i^{\ast}}$'' to obtain
\begin{align*}
  &\quad\E\mathbf{d}^{r}_{\rm H}\bigg(\int^{t}_{t_{0}}\int_{\|z\|\geq1}g(t,s)K(s-,z)N(\dd s,\dd z)+\int^{t}_{t_{0}}\int_{0<\|z\|<1}g(t,s)K(s-,z)\tilde{N}(\dd s,\dd z),\{\0\}\bigg) \\
  &\geq\Big(\E\max_{1\leq j\leq n}|Y_{j,i^{\ast};t}|\Big)^{r},
\end{align*}
which by (\ref{4.10}) grows to infinity in the limit as $n\rightarrow\infty$.

Now, suppose that $\overline{\dec}_{\mathscr{F}_{t}}I^{(4,g)}_{t_{0},t}(K)$ is a bounded subset of the $\mathds{L}^{r}\big(\Omega,\mathscr{F}_{t},\PP;\mathds{R}^{d}\big)$ space. Then, the inequality (\ref{4.8}) (with $p=r$) would imply that so is $\overline{\dec}_{\mathscr{F}_{t}}\big(I^{(3,g)}_{t_{0},t}(K)+I^{(4,g)}_{t_{0},t}(K)\big)$ by Lemma \ref{lem:2} -- a contradiction. Therefore, $\overline{\dec}_{\mathscr{F}_{t}}I^{(4,g)}_{t_{0},t}(K)$ is not a bounded subset of $\mathds{L}^{r}\big(\Omega,\mathscr{F}_{t},\PP;\mathds{R}^{d}\big)$ (and hence $\mathds{L}^{2}\big(\Omega,\mathscr{F}_{t},\PP;\mathds{R}^{d}\big)$).
\end{proof}

A few remarks are in order. First, the conclusion of Theorem \ref{thm:4} is true irrespective of the kernel $g$, i.e., it is applicable to usual set-valued semimartingales (with $g\equiv1$) as well. The immediate implication is that set-valued stochastic integrals with respect to an infinite-variation Poisson random measure can be unbounded like their It\^{o}-type counterparts. Besides, the argument (\ref{4.10}) would fail if one were to consider the stability index $\alpha\in(0,1]$. Note that if $\alpha<1$, the random measure $N$ has finite variation (albeit infinite activity), for which integrable boundedness has been proved; if $\alpha=1$, $N$ does have infinite variation, in which case (\ref{4.10}) is inconclusive. Moreover, in the Gaussian case, the analog of (\ref{4.10}) is nothing but a direct consequence of Slepian's lemma ([Slepian, 1962] \cite{S1}), where $r$ is allowed to be arbitrarily large, which argument was employed in [Michta, 2015] \cite{M5} for proving the unboundedness of $\overline{\dec}_{\mathscr{F}_{t}}I^{(2,1)}_{t_{0},t}(K)$ (as stated in Theorem \ref{thm:3}) upon setting $r=2$. However, it is somewhat cumbersome to employ a Slepian-type inequality for $\alpha$-stable random vectors, such as the one given in [Samorodnitsky and Taqqu, 1993] \cite{ST1}, following the idea of proof in [Michta, 2015] \cite{M5}, to reach the conclusion of Theorem \ref{thm:4}; in fact, the proposed condition \text{ibid.} turns out to be too strong compared to its Gaussian counterpart in [Slepian, 1962] \cite{S1}, failing to establish a useful connection in our context; some details are given in \ref{Appx}.

Theorem \ref{thm:3} and Theorem \ref{thm:4} give rise to plenty of room for constructing integrably bounded set-valued integrands whose set-valued (convoluted) It\^{o} integral and set-valued (convoluted) integral with respect to an infinite-variation Poisson random measure ($q\in\{2,4\}$ in Definition \ref{def:3}) are not integrably bounded in general. For $q=4$, specifically, the condition (\ref{4.2}) implies that such an integrand could contain orthogonal sequences of the $\mathds{L}^{2}\big([t_{0},t);\mathds{R}\big)$ space (compare [Michta, 2015, Remark 3.12] \cite{M5} for $q=2$). We give two simple examples below in $d=1$ dimension.

\medskip

\textsl{Example 4.}\quad Let $K=\breve{S}^{2}_{\mathcal{P}(\mathbb{F})}([0,1])$ and $g(T,s)=T-s$ (Riemann--Liouville) for $s\in[0,T)$. Then, $\card K=2^{\mathfrak{c}}$ and for $q=2$ the set-valued convoluted It\^{o} integral
\begin{equation*}
  \int^{T}_{0}(T-s)[0,1]\dd W_{s}:=\int^{T}_{0}(T-s)\breve{S}^{2}_{\mathcal{P}(\mathbb{F})}([0,1])(s-)\dd W_{s}
\end{equation*}
satisfies $\E\mathbf{d}^{2}_{\rm H}\big(\int^{T}_{0}(T-s)[0,1]\dd W_{s},\{\0\}\big)=\infty$ by Theorem \ref{thm:3}. This can also be viewed as a special instance of [Michta, 2015, Remark 3.12] \cite{M5}.

\medskip

\textsl{Example 5.}\quad Let $K=z\breve{S}^{1}_{\mathcal{P}(\mathbb{F})}([0,1])$ for $0<\|z\|<1$ and still $g(T,s)=T-s$ for $s\in[0,T)$. For $q=4$, consider the symmetric $3/2$-stable process (\text{a.k.a.} the Holtsmark process), whose L\'{e}vy measure is $\nu(\dd z)=C/|z|^{5/2}\dd z$, $z\in\mathds{R}\setminus\{0\}$, with $C>0$ and $\int_{0<|z|<1}|z|\nu(\dd z)=\infty$. Then, the assumptions in Theorem \ref{thm:4} are fulfilled. Based on the class of counterexamples constructed in its proof, for the set-valued convoluted integral with respect to the compensated Poisson random measure
\begin{equation*}
  \int^{T}_{0}\int_{0<|z|<1}(T-s)z[0,1]\tilde{N}(\dd s,\dd z):=\int^{T}_{0}\int_{0<|z|<1}(T-s)z\breve{S}^{2}_{\mathcal{P}(\mathbb{F})}([0,1])(s-)\tilde{N}(\dd s,\dd z),
\end{equation*}
we also have $\E\mathbf{d}^{2}_{\rm H}\big(\int^{T}_{0}\int_{0<|z|<1}(T-s)[0,1]z\tilde{N}(\dd s,\dd z),\{\0\}\big)=\infty$.

\medskip

We close this section with the following theorem to summarize some basic properties of the set-valued convoluted stochastic integrals. They can be readily verified from Definition \ref{def:3}, Theorem \ref{thm:2}, the properties of the decomposable hull and the collections of measurable selectors, and the Eberlein--\v{S}mulian theorem, along the lines of [Kisielewicz, 2020b, Chapter V Theorem 5.2.1] \cite{K8}.

\begin{theorem}\label{thm:5}
Let $K$ and $H$ be nonempty subsets of the $\mathds{L}^{p}(D_{q}\times\Omega,\mathcal{P}(\mathbb{F}),M_{q}\times\PP;E_{q})$ space. Then the following assertions hold for any $q\in\{1,2,3,4\}$, any given suitable kernel $g$, and any fixed $0\leq t_{0}<t\leq T$: \medskip\\
(i) $\int^{t}_{t_{0}}g(t,s)\cl_{\mathds{L}^{p}}K(s-)\dd s=\int^{t}_{t_{0}}g(t,s)K(s-)\dd s$, $\int^{t}_{t_{0}}g(t,s)\cl_{\mathds{L}^{p}}K(s-)\dd W_{s}=\int^{t}_{t_{0}}g(t,s)K(s-)\dd W_{s}$, \\ $\int^{t}_{t_{0}}\int_{\|z\|\geq1}g(t,s)\cl_{\mathds{L}^{p}}K(s-,z)N(\dd s,\dd z)=\int^{t}_{t_{0}}\int_{\|z\|\geq1}g(t,s)K(s-,z)N(\dd s,\dd z)$, \\ and $\int^{t}_{t_{0}}\int_{0<\|z\|<1}g(t,s)\cl_{\mathds{L}^{p}}K(s-,z)\tilde{N}(\dd s,\dd z)=\int^{t}_{t_{0}}\int_{0<\|z\|<1}g(t,s)K(s-,z)\tilde{N}(\dd s,\dd z)$, all $\PP$-a.s.; \\
(ii) $\int^{t}_{t_{0}}g(t,s)\overline{\co}_{\mathds{L}^{p}}K(s-)\dd s=\overline{\co}\int^{t}_{t_{0}}g(t,s)K(s-)\dd s$, $\int^{t}_{t_{0}}g(t,s)\overline{\co}_{\mathds{L}^{p}}K(s-)\dd W_{s}=\overline{\co}\int^{t}_{t_{0}}g(t,s)K(s-)\dd W_{s}$, \\ $\int^{t}_{t_{0}}\int_{\|z\|\geq1}g(t,s)\overline{\co}_{\mathds{L}^{p}}K(s-,z)N(\dd s,\dd z)=\overline{\co}\int^{t}_{t_{0}}\int_{\|z\|\geq1}g(t,s)K(s-,z)N(\dd s,\dd z)$, \\ and $\int^{t}_{t_{0}}\int_{0<\|z\|<1}g(t,s)\overline{\co}_{\mathds{L}^{p}}K(s-,z)\tilde{N}(\dd s,\dd z)=\overline{\co}\int^{t}_{t_{0}}\int_{0<\|z\|<1}g(t,s)K(s-,z)\tilde{N}(\dd s,\dd z)$, all $\PP$-a.s.; \\
(iii) as long as $I^{(q,g)}_{t_{0},t}(K)$ and $I^{(q,g)}_{t_{0},t}(H)$ are bounded subsets of $\mathds{L}^{p}(\Omega,\mathscr{F}_{t},\PP;\mathds{R}^{d})$, $\int^{t}_{t_{0}}g(t,s)(K+H)(s-)\dd s=\int^{t}_{t_{0}}g(t,s)K(s-)\dd s+\int^{t}_{t_{0}}g(t,s)H(s-)\dd s$, $\PP$-a.s., and the same additivity property holds for the other three types of integrals.
\end{theorem}

\medskip

\section{Indefinite integrals}\label{sec:5}

With Definition \ref{def:2} and Definition \ref{def:3} of the four types of set-valued convoluted stochastic integrals (as integral functionals) at fixed time points, we now consider the integrals at indefinite time points, i.e., as set-valued stochastic processes.

\begin{definition}\label{def:4}
In the order of $q\in\{1,2,3,4\}$, for a suitable kernel $g$, the indefinite set-valued stochastic integrals corresponding to Definition \ref{def:3} are defined to be the time-indexed collections
\begin{align*}
  \int^{\cdot}_{0}g(\cdot,s)K(s-)\dd s&\equiv\bigg(\int^{t}_{0}g(t,s)K(s-)\dd s\bigg)_{t\in[0,T]},\\
  \int^{\cdot}_{0}g(\cdot,s)K(s-)\dd W_{s}&\equiv\bigg(\int^{t}_{0}g(t,s)K(s-)\dd W_{s}\bigg)_{t\in[0,T]},\\
  \int^{\cdot}_{0}\int_{\|z\|\geq1}g(\cdot,s)K(s-,z)N(\dd s,\dd z)&\equiv\bigg(\int^{t}_{0}\int_{\|z\|\geq1}g(t,s)K(s-,z)N(\dd s,\dd z)\bigg)_{t\in[0,T]},\\
  \int^{\cdot}_{0}\int_{0<\|z\|<1}g(\cdot,s)K(s-,z)\tilde{N}(\dd s,\dd z)&\equiv\bigg(\int^{t}_{0}\int_{0<\|z\|<1}g(t,s)K(s-,z)\tilde{N}(\dd s,\dd z)\bigg)_{t\in[0,T]}.
\end{align*}
\end{definition}

Unlike their non-convoluted counterparts (with $g\equiv1$), set-valued convoluted stochastic integrals do not belong to the category of set-valued semimartingales (compare the general construction in [Malinowski, 2013] \cite{M1}). As in the single-valued case, for nonempty subsets $K\subset\breve{\mathds{L}}^{p}\big(D_{q}\times\Omega,\mathcal{P}(\mathbb{F}),M_{q}\times\PP;E_{q}\big)$, the integrals exhibit long-range dependence if the kernel $g$ satisfies $\lim_{s\nearrow t}g(t,s)=0$, $\forall t\in(0,T]$ or short-range dependence if $\lim_{s\nearrow t}g(t,s)=\pm\infty$, or $g$ is singular. Indeed, taking $q=2$ for instance, suppose that there exists a square-integrable process $h\in K$ such that $h\neq\0$, $\mathrm{Leb}_{[0,T]}\times\PP$-a.e.; then, one has
\begin{align*}
  \mathrm{Cov}\bigg(\int^{t}_{0}g(t,s)h(s-)\dd W_{s},\int^{t+u}_{0}g(t+u,s)h(s-)\dd W_{s}\bigg)&=\int^{t}_{0}g(t,s)g(t+u,s)h^{2}(s-)\dd s\\
  &=\mathrm{Var}\int^{t}_{0}g(t,s)h(s-)\dd W_{s}+O(u^{\varpi}),\quad u>0,
\end{align*}
where $\varpi\in(0,1)$ and $\varpi>1$ for $\lim_{s\nearrow t}g(t,s)=\pm\infty$ and $\lim_{s\nearrow t}g(t,s)=0$, $\forall t$, respectively.

Nonetheless, in the case of a singular kernel $g$ (as is used for short-range dependence), as long as $\nu(\mathds{R}^{d}\setminus\{\0\})\neq0$, there is a risk that the indefinite integrals with respect to the Poisson random measure are explosive even if the integrand $K$ is a bounded subset. The following theorem concretizes this aspect.

\begin{theorem}\label{thm:6}
Let $K$ be a nonempty closed bounded subset of the $\breve{\mathds{L}}^{p}\big([0,T]\times(\mathds{R}^{d}\setminus\{\0\})\times\Omega,\mathcal{P}(\mathbb{F}), \mathrm{Leb}_{[0,T]}\times\nu\upharpoonright_{\{\|z\|\geq1\}}\times\PP;\mathds{R}^{d}\big)$ and the $\breve{\mathds{L}}^{p}\big([0,T]\times(\mathds{R}^{d}\setminus\{\0\})\times\Omega,\mathcal{P}(\mathbb{F}), \mathrm{Leb}_{[0,T]}\times\nu\upharpoonright_{\{0<\|z\|<1\}}\times\PP;\mathds{R}^{d}\big)$ spaces, respectively, and assume that there exists $h\in K$ with $h\neq\0$ a.e. Then, if the kernel $g$ satisfies the singularity property that $\lim_{s\nearrow t}|g(t,s)|=\infty$ for every $t\in(0,T]$, it happens with positive probability that there exists time $t\in(0,T]$ such that, respectively, $\mathbf{d}_{\rm H}\big(\int^{t}_{0}\int_{\|z\|\geq1}g(t,s)K(s-,z)N(\dd s,\dd z),\{\0\}\big)=\infty$ and $\mathbf{d}_{\rm H}\big(\int^{t}_{0}\int_{0<\|z\|<1}g(t,s)K(s-,z)\tilde{N}(\dd s,\dd z),\{\0\}\big)=\infty$.
\end{theorem}

\begin{proof}
We only consider the case $q=3$; the case $q=4$ can be proved in a similar way. Take $p=1$ without loss of generality.

We observe that for every $t\in(0,T]$,
\begin{align}\label{5.1}
  &\quad\int^{t}_{0}\int_{\|z\|\geq1}g(t,s)h(s-,z)N(\dd s,\dd z)-\int^{t-}_{0}\int_{\|z\|\geq1}g(t-,s)h(s-,z)N(\dd s,\dd z) \nonumber\\
  &=\int^{t-}_{0}\int_{\|z\|\geq1}(g(t,s)-g(t-,s))h(s-,z)N(\dd s,\dd z)+\int^{t}_{t-}\int_{\|z\|\geq1}g(t,s)h(s-,z)N(\dd s,\dd z).
\end{align}
By the continuity of $g:\{(t,s):t\in(0,T),s\in(0,t)\}\mapsto\mathds{R}$ (Definition \ref{def:1}) it is clear that the first (limiting) integral vanishes $\PP$-\text{a.s.} (by (\ref{2.3})), while for the second note that
\begin{equation}\label{5.2}
  \bigg\|\int^{t}_{t-}\int_{\|z\|\geq1}g(t,s)h(s-,z)N(\dd s,\dd z)\bigg\|=|g(t,t-)|\bigg\|\int_{\|z\|\geq1}h(t-,z)N(\{t\},\dd z)\bigg\|.
\end{equation}

As long as $\nu(\mathds{R}^{d}\setminus\{\0\})\neq0$, one familiarly has that $\PP\big\{\int_{\|z\|\geq1}h(t-,z)N(\{t\},\dd z)>0,\exists t\in[0,T]\big\}>0$, but by (\ref{5.1}) and (\ref{5.2}) along with the stated singularity property this means exactly that the probability of
\begin{equation*}
  \bigg\{\bigg\|\int^{t}_{0}\int_{\|z\|\geq1}g(t,s)h(s-,z)N(\dd s,\dd z)\bigg\|=\infty,\;\exists t\in(0,T]\bigg\}
\end{equation*}
is positive, as $|g(t,t-)|=\infty$ $\forall t$. In consequence, let $t$ be the time at which the last result holds, which has positive probability. Then we have
\begin{equation*}
  \mathbf{d}_{\rm H}\bigg(\int^{t}_{0}\int_{\|z\|\geq1}g(t,s)K(s-,z)N(\dd s,\dd z),\{\0\}\bigg)\geq\bigg\|\int^{t}_{0}\int_{\|z\|\geq1}g(t,s)h(s-,z)N(\dd s,\dd z)\bigg\|=\infty,
\end{equation*}
as desired.
\end{proof}

Immediately following Theorem \ref{thm:6} is the next corollary which is of more interest for purely discontinuous models.

\begin{corollary}\label{cor:1}
Assuming the setting of Theorem \ref{thm:6}, if the kernel $g$ satisfies the singularity property that $\lim_{s\nearrow t}|g(t,s)|=\infty$, for every $t\in(0,T]$, and $\nu(\{0<\|z\|<1\})=\infty$, then, $\PP$-a.s., there exists time $t\in(0,T]$ such that $\mathbf{d}_{\rm H}\big(\int^{t}_{0}\int_{0<\|z\|<1}g(t,s)K(s-,z)\tilde{N}(\dd s,\dd z),\{\0\}\big)=\infty$.
\end{corollary}

\begin{proof}
It suffices to note that with $\nu$ being infinite the process $\int^{\cdot}_{0}\int_{0<\|z\|<1}h(s-,z) N(\dd s,\dd z)$ has infinitely many jumps over $(0,T]$, and hence $\PP\big\{\|\int_{0<\|z\|<1}h(t-,z)\tilde{N}(\{t\},\dd z)\|>0,\exists t\in(0,T]\big\}=1$.
\end{proof}

We now give two examples in $d=1$ dimension to illustrate the defined set-valued convoluted stochastic integrals alongside mentions of potential applications.

\medskip

\textsl{Example 6.}\quad Let $g$ be the Riemann--Liouville kernel (refer to Example 2), and set $K_{1}=S^{1}_{\mathcal{P}(\mathbb{F})}([\mu_{1},\mu_{2}])$, $K_{2}=\overline{\co}_{\mathds{L}^{2}}\{\sigma_{1},\sigma_{2}\}$ (which is not decomposable), $K_{3}=\{z\}$ and $K_{4}=\{0\}$, for additional parameters $\mu_{2}>0>\mu_{1}$ and $\sigma_{2}>\sigma_{1}>0$. Consider the set-valued process
\begin{equation*}
  X^{(g)}_{t}=\cl_{\mathds{L}^{1}}\bigg(X_{0}+\int^{t}_{0}[\mu_{1},\mu_{2}]\dd s+\int^{t}_{0}\frac{(t-s)^{\beta-1}}{\Gf(\beta)}\bigg(\overline{\co}_{\mathds{L}^{2}}\{\sigma_{1},\sigma_{2}\}\dd W_{s}+\int_{\mathds{R}\setminus\{0\}}zN(\dd s,\dd z)\bigg)\bigg),\quad t\in(0,T],
\end{equation*}
where $X_{0}\in\mathrm{Cl}(\mathds{R})$ and $N$ is of finite activity. In this case, the (non-convoluted) Aumann integral is in the set-valued sense (Definition \ref{def:3}), the (convoluted) It\^{o} integral is in the quasi-set-valued sense that it equals $[\sigma_{1},\sigma_{2}]\times\int^{t}_{0}(t-s)^{\beta-1}/\Gf(\beta)\dd W_{s}$, and the integral with respect to the Poisson random measure is single-valued. With this construction, $X^{(g)}_{t}$ is integrably bounded\footnote{For the It\^{o} integral, if the integrand is a finite set it is clearly integrably bounded, and taking the closed convex hull will indeed not affect such boundedness; details for the general case can be found in the proof of Theorem \ref{thm:7}.} for every $t\in[0,T]$. Besides, if $\nu\not\equiv0$ and $\beta<1$ the conditions in Theorem \ref{thm:6} hold and the sample paths of $X^{(g)}$ are explosive over time with probability in $(0,1)$. For instance, $X^{(g)}$ can be used to model the dynamics of returns from a risky asset which allows ambiguity in the persistence (drift) and volatility, incorporates jumps, and is also able to capture short- or long-term memory depending on the value of $\beta$.

\medskip

\textsl{Example 7.}\quad Let $g$ be the product of the exponential kernel and the Riemann--Liouville kernel (see Example 1 and Example 2) and set $K_{1}=S^{1}_{\mathcal{P}(\mathbb{F})}([\mu_{1},\mu_{2}])$, $K_{2}=\{0\}$, $K_{3}=S^{1}_{\mathcal{P}(\mathbb{F})}([\gamma_{1},\gamma_{2}])$, and $K_{4}=\overline{\co}_{\mathds{L}^{2}}\{\gamma_{1},\gamma_{2}\}$ (which is not decomposable) for additional parameters $\mu_{2}>0>\mu_{1}$ and $\gamma_{2}>\gamma_{1}>0$. Then, consider the set-valued process\footnote{In general, the Minkowski sum of set-valued stochastic integrals cannot be coalesced into the integral against the usual sum of integrators; a complementation procedure will be necessary for that purpose; see [Malinkowski, 2013, Theorem 3.3] \cite{M1}.}
\begin{align*}
  X^{(g)}_{t}&=\cl_{\mathds{L}^{1}}\bigg(X_{0}+\int^{t}_{0}\frac{e^{-\kappa(t-s)}(t-s)^{\beta-1}}{\Gf(\beta)}[\mu_{1},\mu_{2}]\dd s+\int^{t}_{0}\int_{|z|\geq1}\frac{e^{-\kappa(t-s)}(t-s)^{\beta-1}}{\Gf(\beta)}[\gamma_{1},\gamma_{2}]zN(\dd s,\dd z)\\
  &\qquad+\int^{t}_{0}\int_{0<|z|<1}\frac{e^{-\kappa(t-s)}(t-s)^{\beta-1}}{\Gf(\beta)} \overline{\co}_{\mathds{L}^{2}}\{\gamma_{1},\gamma_{2}\}z\tilde{N}(\dd s,\dd z)\bigg),\quad t\in(0,T],
\end{align*}
where $X_{0}\in\mathrm{Cl}(\mathds{R})$ but $N$ has infinite variation. In this case both the Aumann integral and the integral with respect to the Poisson random measure are convoluted and in the set-valued sense (Definition \ref{def:3}), whereas the integral with respect to the compensated measure (also convoluted) is in the quasi-set-valued sense. This guarantees that $X^{(g)}_{t}$ is integrably bounded for every $t\in[0,T]$. Also, if $\beta<1$, then by Corollary \ref{cor:1} the sample paths of $X^{(g)}$ are explosive over $(0,T]$ with probability 1. In practice, such a process can be used to model the microstructure of log-volatility allowing the reversion level and the jump volatility to be ambiguous, while also capturing potential short-range dependence (if $\beta<1$). However, due to the path explosiveness simulation techniques need to be handled with care.

\medskip

\section{Monotone constructions}\label{sec:6}

In some financial applications of set-valued processes, such as the range of asset prices with model uncertainty (recall Section \ref{sec:1} or [Liang and Ma, 2020] \cite{LM}) and monotonically changing indecisiveness (e.g., see [Xia, 2023+, Section 2] \cite{X}), it is a priori known that the capacity of the sets is a monotone function of time. Instead of contemplating the subset $K$ in order to achieve such effects, we consider two monotone operators acting on the indefinite integrals in Definition \ref{def:4}, which are useful for directly constructing increasing or decreasing processes from convoluted stochastic integral dynamics.

In this section, we consider a set-valued process $X^{(g)}$ of the following general form:
\begin{align}\label{6.1}
  X^{(g)}_{t}&=\cl_{\mathds{L}^{1}}\bigg(X_{0}+\int^{t}_{0}g(t,s)K_{1}(s-)\dd s+\int^{t}_{0}g(t,s)\overline{\co}_{\mathds{L}^{2}}K_{2}(s-)\dd W_{s} \nonumber\\
  &\qquad+\int^{t}_{0}\int_{\|z\|\geq1}g(t,s)K_{3}(s-,z)N(\dd s,\dd z)+\int^{t}_{0}\int_{0<\|z\|<1}g(t,s)\overline{\co}_{\mathds{L}^{2}}K_{4}(s-,z)\tilde{N}(\dd s,\dd z)\bigg),
\end{align}
where $K_{1}$ and $K_{3}$ are closed convex decomposable subsets\footnote{This will ensure that the associated set-valued integrand processes are convex-valued (in $E_{q}$), because in that case their measurable selectors form convex subsets (of $\mathds{L}^{p}$). See, e.g., [Kisielewicz, 2020b, Chapter II Corollary 2.3.3] \cite{K8}.} of the corresponding $\breve{\mathds{L}}^{1}$ spaces ($p=1$ now), while $K_{2}$ and $K_{4}$ are finite (i.e., $1\leq\mathrm{card}K<\aleph$, hence non-decomposable) subsets of the corresponding $\breve{\mathds{L}}^{2}$ spaces (refer to Definition \ref{def:2} and Definition \ref{def:3}), and $N$ is possibly of infinite variation.

Based on (\ref{6.1}), we define two processes
\begin{equation*}
  X^{(g)\downarrow}_{t}(\omega):=\bigcap_{s\in[0,\tau(\omega)\wedge t]}X^{(g)}_{s}(\omega),\quad(t,\omega)\in[0,T]\times\Omega,
\end{equation*}
where
\begin{equation*}
  \tau(\omega):=\inf\Bigg\{t\in[0,T]:\mathrm{card}\bigcap_{s\in[0,t]}X^{(g)}_{s}(\omega)=1\Bigg\},\quad\omega\in\Omega,
\end{equation*}
is the first time when the intersection yields a singleton (of $\mathds{R}^{d}$), and
\begin{equation*}
  X^{(g)\uparrow}_{t}(\omega):=\overline{\co}\bigcup_{s\in[0,t]}X^{(g)}_{s}(\omega),\quad(t,\omega)\in[0,T]\times\Omega.
\end{equation*}
By construction $X^{(g)\downarrow}$ and $X^{(g)\uparrow}$ are set-decreasing and set-increasing processes, respectively, in the sense that $X^{(g)\downarrow}_{s}\supset X^{(g)\downarrow}_{t}$ and $X^{(g)\uparrow}_{s}\subset X^{(g)\uparrow}_{t}$ for any $0\leq s\leq t\leq T$.

To study the integrability and explosiveness of $X^{(g)\downarrow}$ and $X^{(g)\uparrow}$, we will need the following lemma, which is a result of the separability of the probability space.

\begin{lemma}\label{lem:5}
Let $0\leq t_{0}<t\leq T$ be fixed. For $q\in\{1,2,3,4\}$, let $K_{q}\in\mathcal{S}(\breve{\mathds{L}}^{p}(D_{q}\times\Omega,\mathcal{P}(\mathbb{F}),M_{q}\times\PP;E_{q}))$ and let $g$ be a suitable kernel. Then, there is a sequence $\{f_{q,n}:n\in\mathds{N}_{++}\}\subset K_{q}$ such that the following hold $\PP$-a.s.:\\
$\int^{t}_{t_{0}}g(t,s)K_{1}(s-)\dd s=\cl\big\{\int^{t}_{t_{0}}g(t,s)f_{1,n}(s-)\dd s:n\in\mathds{N}_{++}\big\}$;\\ $\int^{t}_{t_{0}}g(t,s)K_{2}(s-)\dd W_{s}=\cl\big\{\int^{t}_{t_{0}}g(t,s)f_{2,n}(s-)\dd W_{s}:n\in\mathds{N}_{++}\big\}$;\\ $\int^{t}_{t_{0}}\int_{\|z\|\geq1}g(t,s)K_{3}(s-,z)N(\dd s,\dd z)=\cl\big\{\int^{t}_{t_{0}}\int_{\|z\|\geq1}g(t,s)f_{3,n}(s-,z)N(\dd s,\dd z):n\in\mathds{N}_{++}\big\}$;\\
$\int^{t}_{t_{0}}\int_{0<\|z\|<1}g(t,s)K_{4}(s-,z)\tilde{N}(\dd s,\dd z)=\cl\big\{\int^{t}_{t_{0}}\int_{0<\|z\|<1}g(t,s)f_{4,n}(s-,z)\tilde{N}(\dd s,\dd z):n\in\mathds{N}_{++}\big\}$.
\end{lemma}

\begin{proof}
We only prove the case $q=4$; the others can be considered similarly. Take $p=2$. According to assertion (iv) of Theorem \ref{thm:2}, note that there is a sequence $\{f_{4,n}:n\in\mathds{N}_{++}\}$ of $K_{4}$ (shortened as $K$ hereinafter) such that $\cl_{\mathds{L}^{2}}I^{(4,g)}_{t_{0},t}(gK)=\cl_{\mathds{L}^{2}}\big\{I^{(4,g)}_{t_{0},t}(gf_{4,n}):n\in\mathds{N}_{++}\big\}$. This implies by the property of decomposable hulls that $\overline{\dec}_{\mathscr{F}_{t}}I^{(4,g)}_{t_{0},t}(gK)=\overline{\dec}_{\mathscr{F}_{t}}\big\{I^{(4,g)}_{t_{0},t}(gf_{4,n}):n\in\mathds{N}_{++}\big\}$. Then, we observe that $\cl_{\mathds{L}^{2}}\big\{I^{(4,g)}_{t_{0},t}(gf_{4,n}):n\in\mathds{N}_{++}\big\}$ constitutes a multifunction from $\Omega$ to $\mathrm{Cl}(\mathds{R}^{d})$ with integrable Hausdorff distance, and by Theorem \ref{thm:1} and Definition \ref{def:3} it follows that
\begin{equation*}
  S^{2}_{\mathscr{F}_{t}}\bigg(\int^{t}_{t_{0}}\int_{0<\|z\|<1}g(t,s)K(s-,z)\tilde{N}(\dd s,\dd z)\bigg)=S^{2}_{\mathscr{F}_{t}}\big(\cl_{\mathds{L}^{2}}\big\{I^{(4,g)}_{t_{0},t}(gf_{4,n}):n\in\mathds{N}_{++}\big\}\big),
\end{equation*}
from where we apply the fundamental property of the collection of measurable selectors to conclude.
\end{proof}

Using the monotone constructions, to require integrability convenience comes at the cost of two additional conditions, which however can be easily justified in practice. The next theorem gives the results.

\begin{theorem}\label{thm:7}
Based on (\ref{6.1}), the set-valued process $X^{(g)\downarrow}$ is $\mathbb{F}$-non-anticipating, integrably bounded (for fixed time), and has $\PP$-\text{a.s.} right-continuous sample paths. Furthermore, assuming that
\begin{equation}\label{6.2}
  \E\sup_{s\in[0,t]}\mathbf{d}_{\rm H}(K_{1}(s-),\{\0\})<\infty,\quad\E\sup_{s\in[0,t]}\int_{\|z\|\geq1}\mathbf{d}_{\rm H}(K_{3}(s-,z),\{\0\})\nu(\dd z)<\infty,\quad\forall t\in[0,T],
\end{equation}
if the kernel $g$ is nonsingular, then the same properties hold for the process $X^{(g)\uparrow}$.
\end{theorem}

\begin{proof}
By Definition \ref{def:3}, each set-valued convoluted stochastic integral is $\mathbb{F}$-non-anticipating (not necessarily $\mathcal{P}(\mathbb{F})$-measurable). By assertion (ii) of Theorem \ref{thm:5}, we have that $X^{(g)}_{t}$ for every $t\in[0,T]$ must be compact convex-valued in $\mathds{R}^{d}$, $\PP$-a.s., because the integral functionals are bounded subsets. Applying Lemma \ref{lem:5}, then there is a sequence $\{\phi^{(g)}_{n}:n\in\mathds{Z}_{++}\}$ of $d$-dimensional $\mathbb{F}$-non-anticipating convoluted L\'{e}vy--It\^{o} processes valued in $\mathds{R}^{d}$ such that $X^{(g)}_{t}=\cl\{\phi^{(g)}_{n}(t):n\in\mathds{Z}_{++}\}$, $\PP$-a.s., for every $t\in[0,T]$. By the right-continuous paths of $\phi^{(g)}_{n}$'s, we can then write for every $t\in[0,T]$, $\PP$-a.s.,
\begin{equation}\label{6.3}
  X^{(g)\downarrow}_{t}=\bigcap_{s\in[0,\tau\wedge t]\cap\mathds{Q}}\cl\big\{\phi^{(g)}_{n}(s):n\in\mathds{Z}_{++}\big\},\quad X^{(g)\uparrow}_{t}=\overline{\co}\bigcup_{s\in[0,t]\cap\mathds{Q}}\cl\big\{\phi^{(g)}_{n}(s):n\in\mathds{Z}_{++}\big\}.
\end{equation}
Further, for $X^{(g)\downarrow}$ note that $\{\tau\leq t\}=\big\{\card\bigcap_{s\in[0,t]\cap\mathds{Q}}X^{(g)}_{s}\leq1\big\}\in\mathscr{F}_{t}$, so that $\tau$ is an $\mathbb{F}$-stopping time (by right-continuity), and on the other hand, we apply Carath\'{e}odory's theorem to $X^{(g)\uparrow}$, which with (\ref{6.3}) establishes the $\mathcal{B}([0,T])\otimes\mathcal{F}$-measurability of $X^{(g)\downarrow}$ and $X^{(g)\uparrow}$ (as processes). In the same spirit we establish from the adaptivity of $\phi^{(g)}_{n}$'s the $\mathscr{F}_{t}$-measurability of $X^{(g)\downarrow}_{t}$ and $X^{(g)\uparrow}_{t}$, for every $t\in[0,T]$. Therefore, $X^{(g)\downarrow}$ and $X^{(g)\uparrow}$ are both $\mathbb{F}$-non-anticipating processes.

Next, that $X^{(g)\downarrow}$ is integrably bounded is immediate from the fact that $\mathbf{d}_{\rm H}\big(X^{(g)\downarrow}_{t},\{\0\}\big)\leq\mathbf{d}_{\rm H}\big(X_{0},\{\0\}\big)$, $\PP$-a.s., which also means that $\E\mathbf{d}_{\rm H}\big(X^{(g)\downarrow}_{t},\{\0\}\big)<\infty$. For $X^{(g)\uparrow}$, it is first observed that
\begin{equation*}
  \mathbf{d}_{\rm H}\big(X^{(g)\uparrow}_{t},\{\0\}\big)\leq\sup_{s\in[0,t]}\mathbf{d}_{\rm H}\big(X^{(g)}_{s},\{\0\}\big),\quad\PP\text{-a.s.}
\end{equation*}
Then, by the properties of the Hausdorff distance and Starr's corollary to the Shapley--Folkman theorem (see [Starr, 1969] \cite{S2}) as well as [Kisielewicz, 2020b, Chapter V Theorem 5.2.1] \cite{K8} it follows that
\begin{align}\label{6.4}
  \mathbf{d}_{\rm H}\big(X^{(g)\uparrow}_{t},\{\0\}\big)&\leq\mathbf{d}_{\rm H}(X_{0},\{\0\})+\E\sup_{s\in[0,t]}\bigg\{\mathbf{d}_{\rm H}\bigg(\int^{s}_{0}g(s,v)K_{1}(v-)\dd v,\{\0\}\bigg) \nonumber\\
  &\quad+\mathbf{d}_{\rm H}\bigg(\int^{s}_{0}\int_{\|z\|\geq1}g(s,v)K_{3}(v-,z)N(\dd z,\dd v),\{\0\}\bigg) \nonumber\\
  &\quad+(\sqrt{d}+1)\bigg(\mathbf{d}_{\rm H}\bigg(\int^{s}_{0}g(s,v)K_{2}(v-)\dd W_{v},\{\0\}\bigg) \nonumber\\
  &\qquad+\mathbf{d}_{\rm H}\bigg(\int^{s}_{0}\int_{0<\|z\|<1}g(s,v)K_{4}(v-,z)\tilde{N}(\dd z,\dd v),\{\0\}\bigg)\bigg)\bigg\}.
\end{align}
By (\ref{6.2}) we immediately have in (\ref{6.4}) that
\begin{equation*}
  \E\sup_{s\in[0,t]}\mathbf{d}_{\rm H}\bigg(\int^{s}_{0}g(s,v)K_{1}(v-)\dd v,\{\0\}\bigg)\leq\int^{t}_{0}|g(t,s)|\dd s\E\sup_{s\in[0,t]}\mathbf{d}_{\rm H}(K_{1}(s),\{\0\})<\infty,
\end{equation*}
and likewise,
\begin{equation*}
  \E\sup_{s\in[0,t]}\mathbf{d}_{\rm H}\bigg(\int^{s}_{0}\int_{\|z\|\geq1}g(s,v)K_{3}(v-,z)N(\dd z,\dd v),\{\0\}\bigg)<\infty.
\end{equation*}
Besides, given that $g$ is nonsingular, note that by setting $\bar{g}=\sup_{t\in[0,T]}\sup_{s\in[0,t)}|g(t,s)|<\infty$ we have
\begin{equation*}
  \mathbf{d}_{\rm H}\bigg(\int^{t}_{0}g(t,s)K_{2}(s-)\dd W_{s},\{\0\}\bigg)\leq\bar{g}\mathbf{d}_{\rm H}\bigg(\int^{t}_{0}K_{2}(s-)\dd W_{s},\{\0\}\bigg),\quad\PP\text{-a.s.}
\end{equation*}
On the right-hand side, the process $\mathbf{d}_{\rm H}\big(\int^{\cdot}_{0}K_{2}(s-)\dd W_{s},\{\0\}\big)$ is a square-integrable submartingale (see, e.g., [Kisielewicz, 2020b, Chapter V Corollary 5.5.1] \cite{K8}), and so applying Doob's maximal inequality together with H\"{o}lder's inequality yields
\begin{align*}
  \E\sup_{s\in[0,t]}\mathbf{d}_{\rm H}\bigg(\int^{s}_{0}g(s,v)K_{2}(v-)\dd W_{v},\{\0\}\bigg)&\leq2\bar{g}\bigg(\E\mathbf{d}^{2}_{\rm H}\bigg(\int^{t}_{0}K_{2}(s-)\dd W_{s},\{\0\}\bigg)\bigg)^{1/2}\\
  &\leq2\bar{g}\bigg(\int^{t}_{0}\sum_{k}\E\|f_{2,k}(s-)\|^{2}_{\rm F}\dd s\bigg)^{1/2},
\end{align*}
where $f_{2,k}$'s are the elements in $K_{2}$ by construction and the second inequality uses [Kisielewicz, 2020b, Chapter V Corollary 5.4.3] \cite{K8}. Similarly, we can build up to
\begin{align*}
  &\quad\E\sup_{s\in[0,t]}\mathbf{d}_{\rm H}\bigg(\int^{s}_{0}\int_{0<\|z\|<1}g(s,v)K_{4}(v-,z)\tilde{N}(\dd z,\dd v),\{\0\}\bigg) \\
  &\leq2\bar{g}\bigg(\int^{t}_{0}\int_{0<\|z\|<1}\sum_{k}\E\|f_{4,k}(s-,z)\|^{2}\nu(\dd z)\dd s\bigg)^{1/2},
\end{align*}
where $f_{4,k}$'s constitute $K_{4}$. These put back into (\ref{6.4}) show that $\E\mathbf{d}_{\rm H}\big(X^{(g)\uparrow}_{t},\{\0\}\big)<\infty$ for every $t\in[0,T]$.

The right continuity with respect to time can be readily established from the properties of the Hausdorff distance and the right-continuity of $X^{(g)}$.
\end{proof}

In contrast, if $g$ happens to be singular, then the process $X^{(g)\uparrow}$ is explosive as long as the jump component remains, resulting in the following corollary.

\begin{corollary}\label{cor:2}
Suppose that $\nu\not\equiv0$ in (\ref{6.1}). If the kernel $g$ has the singularity property that $\lim_{s\nearrow t}|g(t,s)|=\infty$ for every $t\in(0,T]$, then there exists with positive probability $t\in(0,T]$ such that $\mathbf{d}_{\rm H}\big(X^{(g)\uparrow}_{t},\{\0\}\big)=\infty$ for any $s\in[t,T]$.
\end{corollary}

\begin{proof}
Note that as $X^{(g)\uparrow}$ is set-increasing, $\mathbf{d}_{\rm H}\big(X^{(g)\uparrow},\{\0\}\big)$ is a nondecreasing process. Then we use Theorem \ref{thm:6} and observe that, since the time $t$ of explosion, $\mathbf{d}_{\rm H}\big(X^{(g)\uparrow}_{t},\{\0\}\big)\geq\mathbf{d}_{\rm H}\big(X^{(g)}_{t},\{\0\}\big)=\infty$, $\PP$-a.s.
\end{proof}

The implication of Corollary \ref{cor:2} is that set-increasing constructions for set-valued convoluted integrals with respect to Poisson random measures with singular kernels essentially lead to a killed process where the killed states sit at directed infinity, namely $\infty\mathbf{n}$, where $\mathbf{n}\in\mathds{R}^{d}$ is a unit normal. In other words, the resultant increasing process generally takes values in the extended (Euclidean) space $\bar{\mathds{R}}^{d}$, and upon taking extended values, the increasing process cannot be integrably bounded.

Some examples in one dimension are due as usual.

\medskip

\textsl{Example 8.}\quad Consider the process defined in Example 7, based on which construct the monotone processes $X^{(g)\downarrow}$ and $X^{(g)\uparrow}$. If $\beta\geq1$ then both processes are $\mathbb{F}$-non-anticipating and right-continuous according to Theorem \ref{thm:7}. Also, it is clear that conditions (\ref{6.2}) are well met for the choices of integrands. Hence, $X^{(g)\downarrow}$ and $X^{(g)\uparrow}$ are integrably bounded. However, if $\beta<1$, then $X^{(g)\uparrow}_{t}=\bar{\mathds{R}}$, $\PP$-a.s., for every $t\in(0,T]$, despite that $X_{0}\in\mathrm{Cl}(\mathds{R})$, because $N$ has infinite variation; in such a case the increasing process may turn out to be boring as it is killed the moment after time 0.

\medskip

\textsl{Example 9.}\quad Let $g$ be the Molchan--Golosov kernel (refer to Example 3) and $K_{1}=K_{3}=K_{4}=\{0\}$ while $K_{2}=\overline{\co}_{\mathds{L}^{2}}\{0,1\}$. Then the process
\begin{equation*}
  X^{(g)}_{t}=[0,1]\times\int^{t}_{0}(t-s)^{\beta-1}\;_{2}\mathrm{F}_{1}\bigg(-\beta,\beta-1;\beta;-\frac{t-s}{s}\bigg)\dd W_{s}=:[0,1]\times W^{(\beta)}_{t},\quad t\in[0,T]
\end{equation*}
is merely a collection of $[0,1]$-scaled two-sided fractional Brownian motions, and the monotone processes are easily seen for every $t\in[0,T]$ to be $X^{(g)\downarrow}_{t}=\{0\}$, $\PP$-a.s. and $X^{(g)\uparrow}_{t}=\big[m^{(\beta)}_t{},M^{(\beta)}_{t}\big]$, $\PP$-a.s., where $m^{(\beta)}$ and $M^{(\beta)}$ are the running minimum and maximum, respectively, of the fractional Brownian motion (see [Za\"{\i}di and Nualart, 2003] \cite{ZN}). In this case, all processes are integrably bounded and continuous regardless of the value of $\beta>1/2$.

\medskip

\section{Concluding remarks}\label{sec:7}

In this paper we have defined the convoluted stochastic integral under a square-integrable Volterra-type kernel for a ($p\geq1$)-integrably bounded set-valued stochastic process with the boundedness property (\ref{2.3}), which guarantees that the resultant integral can be well understood and analyzed in an $\mathds{L}^{p}$ space. Volterra-type kernels are deemed general enough to accommodate most practical interests, especially the incorporation of short- or long-range dependence in financial modeling. Based on set-valued convoluted integral functionals (Definition \ref{def:2}), we define the corresponding set-valued integrals as random variables whose collection of measurable selectors equals the closed decomposable hull of these functionals (Definition \ref{def:3}), which matches classical definitions ([Kisielewicz, 2012] \cite{K5} and [Zhang et al., 2013] \cite{ZMO}) at a general level.

Importantly, we have demonstrated (Theorem \ref{thm:4}) that, similar to set-valued It\^{o} integrals, set-valued stochastic integrals with respect to an infinite-variation Poisson random measure need not be integrably bounded with a bounded integrand, depending on the existence of distinct deterministic selectors. These results well complement the findings of [Zhang et al., 2013] \cite{ZMO}, [Michta, 2015] \cite{M5}, and [Kisielewicz, 2020a] \cite{K7}. Based on the proof of Theorem \ref{thm:4}, the intuition is that due to the infinite variation of the integrators, if the integrands exhibit certain levels of separation in the $\mathds{L}^{2}$ space, then it is possible to construct a sequence of elements belonging to the integral ($\PP$-a.s.) whose supremum (which is a decomposable combination) is not (square-)integrable. We remark that the study of such integrals driven by infinite random measures is not reasonably substitutable by that of set-valued It\^{o} integrals due to their crucial role in attaining specific path regularities, even for a semimartingale -- a good example would be the stability index $\alpha\in(0,2)$ of a symmetric stable process.

We have also shown (Theorem \ref{thm:6}) that if the kernel is singular, then the set-valued convoluted stochastic integrals as stochastic processes can take extended vector values in the presence of jumps. Intuitively, this is because each time when there is a discontinuity in the sample paths, the singularity will necessarily generate a point of directed infinity. However, in general this does not affect the integrable boundedness (if it holds without singularity) at generic time, but simulation methods need to be reconsidered for these extended values; e.g., Euler discretization may not be reliably implemented, especially if the process is to model a utility set, in which case there must be a well-defined limiting utility associated with such an index. Furthermore, when set-monotone processes are considered based on a given set-valued convoluted stochastic integral according to the procedures in Section \ref{sec:6}, the coexistence of singular kernels and infinitely active jumps should be avoided in construction as they lead to instantly killed processes that are never integrably bounded.

We conclude the paper by presenting the following general version of a set-valued convoluted stochastic differential inclusion that arises from the main results and that one may expect to work with taking into account all possible shapes of ingredients,
\begin{align}\label{7.1}
  X_{t}&\in\cl_{\mathds{L}^{1}}\bigg(X_{0}+\sum_{k\in\mathbb{K}}\bigg(\int^{t}_{0}g_{k}(t,s)F_{1,k}(s-,X_{s-})\dd s+\int^{t}_{0}g_{k}(t,s)\overline{\co}_{\mathds{L}^{2}}F_{2,k}(s-,X_{s-})\dd W_{s} \nonumber\\
  &\qquad+\int^{t}_{0}\int_{\|z\|\geq1}g_{k}(t,s)F_{3,k}(s-,X_{s-},z)N(\dd s,\dd z)\nonumber\\
  &\qquad+\int^{t}_{0}\int_{0<\|z\|<1}g_{k}(t,s)\overline{\co}_{\mathds{L}^{2}}F_{4,k}(s-,X_{s-},z)\tilde{N}(\dd s,\dd z)\bigg)\bigg),\quad t\in[0,T],
\end{align}
where $g_{k}$'s are suitable Volterra-type kernels (Definition \ref{def:1}) for $k\in\mathbb{K}$, $\mathbb{K}$ being a finite set, $F_{1}:[0,T]\times\mathds{R}^{d}\mapsto\mathrm{Cl}(\mathds{R}^{d})$ and $F_{3}:[0,T]\times\mathds{R}^{d}\times(\mathds{R}^{d}\setminus\{\0\})\mapsto\mathrm{Cl}(\mathds{R}^{d})$ are both multifunctions with measurable selectors in $\breve{S}^{1}(F_{1})$ and $\breve{S}^{2}(F_{3})$ (i.e., $\mathds{L}^{1,2}$ with a finite temporal supremum; see (\ref{2.3})), and $F_{2}$ and $F_{4}$ both consist of finitely many square-integrably bounded functions -- specifically, $F_{2}=\big\{(f_{2,n}:[0,T]\times\mathds{R}^{d}\mapsto\mathds{R}^{d\times m}):\sup\big\{\int_{\mathds{R}^{d}}f^{2}_{2,n}(t,x)\dd x:t\in[0,T]\big\}<\infty\big\}$ and $F_{4}=\big\{(f_{4,n}:[0,T]\times\mathds{R}^{d}\times(\mathds{R}^{d}\setminus\{\0\})\mapsto\mathds{R}^{d}): \sup\big\{\int_{\mathds{R}^{d}}\int_{0<\|z\|<1}f^{2}_{4,n}(t,x,z)\nu(\dd z)\dd x:t\in[0,T]\big\}<\infty\big\}$. Notably, the Poisson random measure can have infinite variation. Other aforementioned remarks aside, studying the differential inclusion (\ref{7.1}) in various applications will also be interesting for further research.

\medskip

\appendix
\gdef\thesection{Appendix}

\section{Attempt to prove Theorem \ref{thm:4} via a Slepian-type inequality}\label{Appx}

A symmetric non-Gaussian infinitely divisible random vector $X$ with values in $\mathds{R}^{n}$ ($n\in\mathds{N}_{++}$) is said to be of type G if its characteristic function takes the form
\begin{equation}\label{A.1}
  \E e^{\langle\ii u,X\rangle}=\exp\bigg(-\int_{\mathds{R}^{n}}\psi\bigg(\frac{\langle u,x\rangle^{2}}{2}\bigg)\eta(\dd x)\bigg),\quad u\in\mathds{R}^{n},
\end{equation}
for some function $\psi$ with completely monotone derivatives on $\mathds{R}_{++}$ (i.e., all derivatives with alternating signs) satisfying that $\psi(0)=\lim_{x\rightarrow\infty}\psi'(x)=0$ and some $\upsigma$-finite measure $\eta$ on $\mathds{R}^{n}$. Type-G infinitely divisible random vectors can be constructed from Gaussian mixtures of infinitely divisible distributions; refer to [Rosi\'{n}ski, 1991, Section 2] \cite{R} for more details.

The lemma below, adapted from [Samorodnitsky and Taqqu, 1993, Theorem 3.1] \cite{ST1}, gives a Slepian-type inequality (see [Samorodnitsky and Taqqu, 1994] \cite{ST2} for some related inequalities) for type-G random vectors, which can be compared to Slepian's lemma ([Slepian, 1962] \cite{S1}) for Gaussian processes.

\begin{lemma}\label{lem:A}
Let $X$ and $Y$ be two type-G infinitely divisible random vectors in $\mathds{L}^{1}(\Omega;\mathds{R}^{n})$, $n\geq2$, with L\'{e}vy measures $\nu_{X}$ and $\nu_{Y}$ and conjugate L\'{e}vy measures $\hat{\nu}_{X}$ and $\hat{\nu}_{Y}$ (see [Samorodnitsky and Taqqu, 1993, Definition 3.2] \cite{ST1}), respectively. If
\begin{equation*}
  \hat{\nu}_{X}\big(y\in\mathds{R}^{n}:\;[(y_{j}-y_{j'})^{2}:j<j']\in A\big)\geq\hat{\nu}_{Y}\big(y\in\mathds{R}^{n}:\;[(y_{j}-y_{j'})^{2}:j<j']\in A\big),
\end{equation*}
for every increasing set $A$ (in the sense that if $a\in A$ and $b\in\mathds{R}^{n(n-1)/2}_{+}$ satisfy $b_{l}\geq a_{l}$, $\forall l\in\mathds{N}\cap[1,n(n-1)/2]$, then $b\in A$), then
\begin{equation*}
  \E\max_{j\in\mathds{N}\cap[1,n]}X_{j}\geq\E\max_{j\in\mathds{N}\cap[1,n]}Y_{j}.
\end{equation*}
\end{lemma}

Let us resume the setting of Step 3 in the proof of Theorem \ref{thm:4}. To invoke Lemma \ref{lem:A}, we have to find a L\'{e}vy measure $\nu_{0}$ on $\mathds{R}^{n}\setminus\{\0\}$, also for an $n$-dimensional type-G random vector as $\bar{\nu}_{i;t}$, such that the associated conjugate measures satisfy
\begin{equation}\label{A.2}
  \hat{\bar{\nu}}_{i^{\ast};t}\big(y\in\mathds{R}^{n}:\;[(y_{j}-y_{j'})^{2}:j<j']\in A\big)\geq\varphi_{n}(\rho)=\hat{\nu}_{0}\big(y\in\mathds{R}^{n}:\;[(y_{j}-y_{j'})^{2}:j<j']\in A\big),
\end{equation}
for every increasing set $A\subset\mathds{R}^{n(n-1)/2}_{+}$ and some decreasing function $\varphi_{n}:\mathds{R}_{+}\mapsto\mathds{R}_{+}$ (see below).

If we choose $\nu_{0}$ for an $n$-dimensional L\'{e}vy process with \text{i.i.d.} components having a type-G distribution, supported on the union of axes $\bigcup^{n}_{j=1}\{[0,\dots,0,y_{j},0,\dots,0]:y_{j}\in\mathds{R}\}\setminus\{\0\}$, $n\geq2$, then any increasing set $A$ as appearing in (\ref{A.2}) with nonzero measure under $\hat{\nu}_{0}$ has to contain a vector with exactly $n-1$ nonzero equal values, say $\rho>0$, and $(n-1)(n-2)/2$ zero values. For example, consider the particular choice $A=\mathds{R}^{n(n-1)/2}_{+}+\{[\rho,\dots,\rho,0,\dots,0]\}$. Then, by the symmetry property,
\begin{equation}\label{A.3}
  \varphi_{n}(\rho)=2\hat{\nu}_{0,1}([\sqrt{\rho},\infty))>0,
\end{equation}
where, similar as before, $\hat{\nu}_{0,1}$ denotes the conjugate L\'{e}vy measure of the first component.

On the other hand, by noting that for symmetric $\alpha$-stable random vectors, the conjugate measure $\hat{\bar{\nu}}_{i^{\ast};t}$ is nothing but a scalar multiple (see [Samorodnitsky and Taqqu, 1993, Example 3.2] \cite{ST1}), from (\ref{4.9}) it follows that (with $A=\mathds{R}^{n(n-1)/2}_{+}+\{[\rho,\dots,\rho,0,\dots,0]\}$ and a positive, only-$\alpha$-depending constant $\tilde{C}_{\alpha}$)
\begin{align}\label{A.4}
  \hat{\bar{\nu}}_{i^{\ast};t}\big(y\in\mathds{R}^{n}:\;[(y_{j}-y_{j'})^{2}:j<j']\in A\big)&=\hat{\bar{\nu}}_{i^{\ast};t}\big(y\in\mathds{R}^{n}:\; (y_{1}-y_{j'})^{2}\geq\rho,\forall j'>1\big) \nonumber\\
  &=\tilde{C}_{\alpha}\sum^{d}_{i=1}\int^{t}_{t_{0}}\int^{\infty}_{0+} \frac{\tilde{C}_{\alpha}\big|g(t,s)b_{1,i^{\ast},i}(s)\big|^{\alpha}}{|y_{1}|^{\alpha+1}} \nonumber\\
  &\qquad\times\mathds{1}_{\big\{y^{2}_{1}\geq\rho b^{2}_{1,i^{\ast},i}(s)\big/\inf_{j'>1}(b_{1,i^{\ast},i}(s)-b_{j',i^{\ast},i}(s))^{2}\big\}}\dd y_{1}\dd s \nonumber\\
  &=\frac{2\tilde{C}_{\alpha}}{\rho^{\alpha/2}}\sum^{d}_{i=1}\int^{t}_{t_{0}}|g(t,s)|^{\alpha} \inf_{j'>1}|b_{1,i^{\ast},i}(s)-b_{j',i^{\ast},i}(s)|^{\alpha}\dd s,
\end{align}
where the last equality follows from rewriting the set in $\hat{\nu}_{i}$ using the two conditions and the symmetry of $\hat{\nu}_{i}$, $i\in\mathds{N}\cap[1,d]$, and which vanishes in the limit as $n\rightarrow\infty$ for any bounded functions $b_{j',i^{\ast},i}$. Comparing (\ref{A.3}) and (\ref{A.4}), it is not possible for the inequality in (\ref{A.2}) to hold for every $n\geq2$. In a way, this shows that the Slepian-type inequality in Lemma \ref{lem:A} is ``too sufficient'' for a reproduction of the argument in the Gaussian case.


\begin{thebibliography}{99}\footnotesize
\bibitem{A1} Applebaum, D. {\sl L\'{e}vy Processes and Stochastic Calculus} (2nd Ed.), Cambridge University Press, Cambridge, UK, 2009.
\bibitem{AF} Ararat, \c{C}. and Feinstein, Z. (2021). Set-valued risk measures as backward stochastic difference inclusion and equations. {\sl Finance and Stochastics}, 25: 43--76.
\bibitem{AM} Ararat, \c{C}. and Ma, J. (2023+). Path-regularity and martingale properties of set-valued stochastic integrals. {\sl arXiv Working Paper}, 2308.13110, 1--47.
\bibitem{AMW} Ararat, \c{C}., Ma, J., and Wu, W. (2023). Set-valued backward stochastic differential equations. {\sl Annals of Applied Probability}, 33: 3418--3448.
\bibitem{A2} Aumann, R.J. (1965). Integrals of set-valued functions. {\sl Journal of Mathematical Analysis and Applications}, 12: 1--12.
\bibitem{B-NBPV} Barndorff-Nielsen, O.E., Benth, F.E., Pedersen, J., and Veraat, A.E.D. (2014). On stochastic integration for volatility modulated L\'{e}vy-driven Volterra processes. {\sl Stochastic Processes and their Applications}, 124: 812--847.
\bibitem{BM} Bender, C. and Marquardt, T. (2008). Stochastic calculus for convoluted L\'{e}vy processes. {\sl Bernoulli}, 14: 499--518.
\bibitem{BG} Blumenthal R.M. and Getoor, R.K. (1961). Sample functions of stochastic processes with stationary independent increments. {\sl Journal of Mathematics and Mechanics}, 10: 493--516.
\bibitem{ER} El Euch, O. and Rosenbaum, M. (2018). Perfect hedging in rough Heston models. {\sl Annals of Applied Probability}, 28: 3813--3856.
\bibitem{F} Fryszkowski, A. {\sl Fixed Point Theory for Decomposable Sets}, Kluwer Academic Publishers, New York, 2004.
\bibitem{GJR} Gatheral, J., Jaisson, T., and Rosenbaum, M. (2018). Volatility is rough. {\sl Quantitative Finance}, 18: 933-949.
\bibitem{HU} Hiai, F. and Umegaki, H. (1977). Integrals, conditional expectations and martingales of multivalued functions. {\sl Journal of Multivariate Analysis}, 7: 149--182.
\bibitem{H} Hult, H. (2003). Approximating some Volterra type stochastic integrals with applications to parameter estimation. {\sl Stochastic Processes and Their Applications}, 105: 1--32.
\bibitem{J} Jost, C. (2006). Transformation formulas for fractional Brownian motion. {\sl Stochastic Processes and their Applications}, 116: 1341--1357.
\bibitem{K1} Kisielewicz, M. {\sl Differential Inclusions and Optimal Control}. Kluwer Academic Publishers, New York, 1991.
\bibitem{K2} Kisielewicz, M. (1993). Properties of solution set of stochastic inclusions. {\sl Journal of Applied Mathematics and Stochastic Analysis}, 6: 217--236.
\bibitem{K3} Kisielewicz, M. (1997). Set-valued stochastic intergrals and stochastic inclutions. {\sl Stochastic Analysis and Applications}, 15: 783--800.
\bibitem{K4} Kisielewicz, M. (2005). Weak compactness of solution sets to stochastic differential inclusions with non-convex right-hand sides. {\sl Stochastic Analysis and Applications}, 23: 871--901.
\bibitem{K5} Kisielewicz, M. (2012). Some properties of set-valued stochastic integrals. {\sl Journal of Mathematical Analysis and Applications}, 388: 984--995.
\bibitem{K6} Kisielewicz, M. {\sl Stochastic Differential Inclusions and Applications}. Springer-Verlag, Berlin, 2013.
\bibitem{K7} Kisielewicz, M. (2020a). Integrable boundedness of set-valued stochastic integrals. {\sl Journal of Mathematical Analysis and Applications}, 481: 123441.
\bibitem{K8} Kisielewicz, M. {\sl Set-Valued Stochastic Integrals and Applications}. Springer Nature Switzerland, Switzerland, 2020b.
\bibitem{KM} Kisielewicz, M. and Michta, M. (2017). Integrably bounded set-valued stochastic integrals. {\sl Journal of Mathematical Analysis and Applications}, 449: 1892--1910.
\bibitem{LL1} Li, J. and Li, S. (2009a). Aumann type set-valued Lebesgue integral and representation theorem. {\sl International Journal of Computational Intelligence Systems}, 2: 83--90.
\bibitem{LL2} Li, J. and Li, S. (2009b). It\^{o} type set-valued stochastic differential equation. {\sl Journal of Uncertain Systems}, 3: 52--63.
\bibitem{LLL} Li, S., Li, J., and Li, X. (2010). Stochastic integral with respect to set-valued square integrable martingales. {\sl Journal of Mathematical Analysis and Applications}, 370: 659--671.
\bibitem{LM} Liang, Z. and Ma, M. (2020). Robust consumption-investment problem under CRRA and CARA utilities with time-varying confidence sets. {\sl Mathematical Finance}, 30: 1035--1072.
\bibitem{L} Lyasoff, A. {\sl Stochastic Methods in Asset Pricing}, MIT Press, Cambridge, MA, 2017.
\bibitem{M1} Malinowski, M.T. (2013). On a new set-valued stochastic integral with respect to semimartingales and its applications. {\sl Journal of Mathematical Analysis and Applications}, 408: 669--680.
\bibitem{M2} Malinowski, M.T. (2015). Set-valued and fuzzy stochastic integral equations driven by semimartingales under Osgood condition. {\sl Open Mathematics}, 13: 106--134.
\bibitem{M3} Marcus, M. B. (1984). Extreme values for sequences of stable random variables. In: Tiago de Oliveira. {\sl Statistical Extremes and Applications}, pp. 311--324, Springer Netherlands, Reidel, Dordrecht.
\bibitem{MP} Marcus, M.B. and Pisier, G. (1984). Characterizations of almost surely continuous $p$-stable random Fourier series and strongly stationary processes. {\sl Acta Mathematica}: 152: 245--301.
\bibitem{M4} Marquardt, T. (2006). Fractional L\'{e}vy processes with an application to long memory moving average processes. {\sl Bernoulli}, 12: 1099--1126.
\bibitem{M5} Michta, M. (2015). Remarks on unboundedness of set-valued It\^{o} stochastic integrals. {\sl Journal of Mathematical Analysis and Applications}, 424: 651--663.
\bibitem{MG} Molchan, G. and Golosov, J. (1969). Gaussian stationary processes with asymptotic power spectrum. {\sl Soviet Mathematics Doklady}, 10: 134--137.
\bibitem{N} Nualart, D. (2003). Stochastic calculus with respect to the fractional Brownian motion and applications. {\sl Contemporary Mathematics}, 336: 3--39.
\bibitem{RW} Ren, J. and Wu, J. (2011). Multi-valued stochastic differential equations driven by Poisson point processes. In: Kohatsu-Higa, A., Privault, N., and Sheu, S.J. (eds). {\sl Stochastic Analysis with Financial Applications. Progress in Probability}, vol. 65, Springer, Basel.
\bibitem{R} Rosi\'{n}ski, J. (1991). On a class of infinitely divisible processes represented as mixtures of Gaussian processes. In: Samorodnitsky, G., Cambanis, S., and Taqqu, M.S. (eds). {\sl Stable Processes and Related Topics}, pp. 27--41, Birkh\"{a}user, Boston.
\bibitem{ST1} Samorodnitsky, G. and Taqqu, M.S. (1993). Stochastic monotonicity and Slepian-type inequalities for infinitely divisible and stable random vectors. {\sl Annals of Probability}, 21: 143--160.
\bibitem{ST2} Samorodnitsky, G. and Taqqu, M.S. (1994). L\'{e}vy measures of infinitely divisible random vectors and Slepian inequalities. {\sl Annals of Probability}, 22: 1930--1956.
\bibitem{S1} Slepian, D. (1962). The one-sided barrier problem for Gaussian noise. {\sl Bell System Technical Journal}, 41: 463--501.
\bibitem{S2} Starr, R.M. (1969). Quasi-equilibria in markets with non-convex preferences. {\sl Econometrica}, 37: 25--38.
\bibitem{TM} Tikanm\"{a}ki, H. and Mishura, Y. (2011). Fractional L\'{e}evy processes as a result of compact interval integral transformation. {\sl Stochastic Analysis and Applications}, 29: 1080--1101.
\bibitem{TT} Todorov, V. and Tauchen, G. (2011). Volatility jumps. {\sl Journal of Business and Economic Statistics}, 29: 356--371.
\bibitem{WX} Wang, L. and Xia, W. (2022). Power-type derivatives for rough volatility with jumps. {\sl Journal of Futures Markets}, 42: 1369--1406.
\bibitem{WT} Wolpert, R.L. and Taqqu, M.S. (2004). Fractional Ornstein--Uhlenbeck L\'{e}vy processes and the Telecom process: Upstairs and downstairs. {\sl Signal Processing}, 85: 1523--1545.
\bibitem{X} Xia, W. (2023+). Optimal consumption--investment problems under time-varying incomplete preferences. {\sl arXiv Working Paper}, 2312.00266, 1--72.
\bibitem{ZN} Za\"{\i}di, N.L. and Nualart, D. (2003). Smoothness of the law of the supremum of the fractional Brownian motion. {\sl Electronic Communications in Probability}, 8: 102--111.
\bibitem{ZLMO} Zhang, J., Li, S., Mitoma, I., and Okazaki, Y. (2009). On the solutions of set-valued stochastic differential equations in M-type 2 Banach spaces. {\sl Tohoku Mathematical Journal}, 61: 417--440.
\bibitem{ZMO} Zhang, J., Mitoma, I., and Okazaki, Y. (2013). Set-valued stochastic integrals with respect to Poisson processes in a Banach space. {\sl International Journal of Approximate Reasoning}, 54: 404--417.
\end{thebibliography}
\end{document}